\documentclass[12pt]{article}

\usepackage{graphicx}
\usepackage{amssymb}
\usepackage{amsmath}
\usepackage{enumerate}
\usepackage{amsfonts}
\usepackage{amsthm}
\usepackage[all,import]{xy}
\usepackage{xcolor}
\usepackage{url}
\makeatletter
\def\url@leostyle{%
  \@ifundefined{selectfont}{\def\UrlFont{\sf}}{\def\UrlFont{\small\ttfamily}}}
\makeatother
\numberwithin{equation}{section}
\usepackage[colorlinks=true, pdfstartview=FitV, linkcolor=black, 
            citecolor=black, urlcolor=black]{hyperref}
\usepackage{bookmark}
\usepackage[titletoc]{appendix} 
\theoremstyle{definition}
\newtheorem{prop}{Proposition}[section]

\newtheorem{theorem}[prop]{Theorem}
\newtheorem{lemma}[prop]{Lemma}
\newtheorem{cor}[prop]{Corollary}
\newtheorem{corollary}[prop]{Corollary}

\newtheorem{defn}[prop]{Definition}
\newtheorem{example}[prop]{Example}
\newtheorem{notation}[prop]{Notation}
\newtheorem{remark}[prop]{Remark}
\newtheorem{image}[prop]{Figure} 

\newcommand{\nc}{\newcommand}
\nc{\DMO}{\DeclareMathOperator}	

\nc{\newnotation}{\nomenclature}
\nc{\wrap}{\cW}
\nc{\Cob}{\mathsf{Cob}}
\nc{\mul}{\mathsf{Mul}}
\nc{\fat}{\mathsf{fat}}
\nc{\cob}{\mathsf{Cob}}
\nc{\coh}{\mathsf{Coh}}
\nc{\idem}{\mathsf{Idem}}
\nc{\sets}{\mathsf{Sets}}
\nc{\near}{\mathsf{near}}
\nc{\sing}{\mathsf{Sing}}
\nc{\symp}{\mathsf{Symp}}
\nc{\perf}{\mathsf{Perf}}
\nc{\ssets}{\mathsf{sSets}}
\nc{\cmpct}{\mathsf{cmpct}}
\nc{\compact}{\mathsf{cmpct}}
\nc{\pwrap}{\mathsf{PWrap}}
\nc{\coder}{\mathsf{Coder}}
\nc{\bimod}{\mathsf{Bimod}}
\nc{\grmod}{\mathsf{GrMod}}
\nc{\spaces}{\mathsf{Spaces}}
\nc{\pwrms}{\mathsf{PWrFuk}_{M,S}}
\nc{\pwrmf}{\mathsf{PWrFuk}_{M,F}}
\nc{\pwrapmf}{\mathsf{PWrFuk}_{M,F}}
\nc{\fuk}{\mathsf{Fukaya}}
\nc{\infwr}{\mathsf{InfWr}}
\nc{\fukaya}{\mathsf{Fukaya}}
\nc{\autml}{\mathsf{Aut}_{M,\Lambda}}
\nc{\fukml}{\mathsf{Fukaya}_{M,\Lambda}}
\nc{\fukmle}{\mathsf{Fukaya}_{M,\Lambda,\epsilon}}
\nc{\fukmod}{\wrfukcompact(M)\modules}
\nc{\lag}{\mathsf{Lag}}
\nc{\lagm}{\lag_M}
\nc{\lago}{\lag^o}
\nc{\lagml}{\lag_{M,\Lambda}} 
\nc{\lagmle}{\lag_{M,\Lambda,\epsilon}}
\nc{\fun}{\mathsf{Fun}}
\nc{\vect}{\mathsf{Vect}}
\nc{\chain}{\mathsf{Chain}}
\nc{\wrfuk}{\mathsf{WrFukaya}}
\nc{\wrfukcompact}{\mathsf{WrFukaya}_{\mathsf{cmpct}}}
\nc{\pwrfuk}{\mathsf{PWrFukaya}}
\nc{\inffuk}{\mathsf{InfFuk}}
\nc{\pwrfukml}{\mathsf{PWrFukaya}_{M,\Lambda}}
\nc{\inffukml}{\mathsf{InfFuk}_{M,\Lambda}}
\nc{\nattrans}{\mathsf{NatTrans}}
\nc{\corres}{\mathsf{Corres}}
\nc{\fukep}{\fukaya_\Lambda(M,\epsilon)}
\nc{\fukepop}{\fukaya_\Lambda(M,\epsilon)^{\op}}
\nc{\lagep}{\lag_\Lambda(M,\epsilon)}
\DMO{\cyl}{cyl} 
\nc{\dbcoh}{D^b\mathsf{Coh}}
\nc{\corr}{\mathsf{Corr}}

\nc{\cat}{\mathsf{Cat}}
\nc{\Cat}{\mathsf{Cat}}
\nc{\ainfty}{\mathsf{A}_\infty}
\nc{\inftycat}{\mathcal{C}\!\operatorname{at}_\infty}
\nc{\Ainftycat}{\mathcal{C}\!\operatorname{at}_{A_\infty}}
\nc{\ainftycat}{\mathcal{C}\!\operatorname{at}_{A_\infty}}
\nc{\stablecat}{\mathcal{C}\!\operatorname{at}_\infty^{\Ex}}

\DMO{\im}{im}
\DMO{\ev}{ev}
\DMO{\inj}{inj}
\DMO{\fib}{fib}
\DMO{\conf}{Conf}
\DMO{\chains}{Chains}
\DMO{\cochains}{Cochains}
\DMO{\cone}{Cone}
\DMO{\ran}{Ran}
\DMO{\rot}{Rot}
\DMO{\leg}{Leg}
\DMO{\imm}{imm}
\DMO{\adj}{adj}
\DMO{\cube}{Cube}
\DMO{\deep}{deep}
\DMO{\back}{back}
\DMO{\front}{front}
\DMO{\flow}{Flow}
\DMO{\floer}{Floer}
\DMO{\maps}{Maps}
\DMO{\exact}{exact}
\DMO{\Decomp}{Decomp}
\DMO{\decomp}{Decomp}
\DMO{\collar}{collar}
\DMO{\yoneda}{Yoneda}
\DMO{\hamspace}{Ham}
\DMO{\sympspace}{Symp}
\DMO{\holomaps}{Holomaps}
\DMO{\comp}{Comp}
\DMO{\crit}{Crit}
\DMO{\test}{{test}}
\DMO{\sign}{sign}
\DMO{\topp}{top}
\DMO{\indx}{Index}
\DMO{\Break}{Break} 
\DMO{\zero}{zero} 
\DMO{\ob}{Ob}
\DMO{\gr}{Gr} 
\DMO{\Gr}{Gr} 
\DMO{\cl}{Cl} 
\DMO{\grlag}{GrLag}
\DMO{\Pin}{Pin}
\DMO{\Graph}{Graph}
\DMO{\pin}{Pin}
\DMO{\gap}{Gap}
\DMO{\Ex}{Ex}
\DMO{\id}{id}
\DMO{\End}{End}
\DMO{\sym}{Sym} 
\DMO{\aut}{Aut}
\DMO{\DK}{DK} 
\DMO{\poly}{poly} 
\DMO{\diff}{Diff}
\DMO{\coll}{coll}
\DMO{\dist}{dist} 
\DMO{\coker}{coker} 
\nc{\kernel}{\ker} 
\DMO{\sspan}{span}
\DMO{\hocolim}{hocolim}	
\DMO{\holim}{holim}
\DMO{\sk}{sk}

\DMO{\ho}{ho}
\DMO{\fin}{fin}
\DMO{\ret}{Ret}
\DMO{\ham}{Ham}
\DMO{\con}{con}
\DMO{\leaf}{leaf}
\DMO{\supp}{supp}
\DMO{\edge}{edge}
\DMO{\colim}{colim}
\DMO{\edges}{edges}
\DMO{\Image}{image}
\DMO{\roots}{roots}
\DMO{\height}{height}
\DMO{\finmod}{FinMod}
\DMO{\leaves}{leaves}
\DMO{\planar}{planar}
\DMO{\vertices}{vertices}

\nc{\lagg}{\lag^{\cG}}
\nc{\iso}{\mathsf{Iso}}
\nc{\Set}{\mathsf{Set}}
\nc{\ass}{\mathsf{ \bf Ass}}
\nc{\Mod}{\mathsf{Mod}}
\nc{\modules}{\mathsf{Mod}}
\nc{\sset}{\mathsf{sSet}}
\nc{\liou}{\mathsf{Liou}}
\nc{\poset}{\mathsf{Poset}}
\nc{\trno}{T^*\RR^n_{\geq 0}}
\nc{\spectra}{\mathsf{Spectra}}
\nc{\tensorfin}{\tensor^{\fin}}
\nc{\lagptg}{\lag_{pt,pt}^{\cG}}
\nc{\Fin}{\mathcal{F}\mathsf{in}}
\nc{\lagnl}{\lag_{N,\Lambda}}
\nc{\lagmlg}{\lag_{M,\Lambda}^{\cG}}
\nc{\lagsplit}{\lag^{\mathsf{split}}}
\nc{\lagktimes}{(\lag^{\dd k})^\times}
\nc{\lagplanar}{\lag^{\times,\planar}}

\nc{\smsh}{\wedge}
\nc{\un}{\underline}
\nc{\xto}{\xrightarrow}
\nc{\xra}{\xto}
\nc{\tensor}{\otimes}
\nc{\del}{\partial}
\nc{\dd}{\diamond}
\nc{\tri}{\triangle}
\nc{\bb}{\Box}
\nc{\into}{\hookrightarrow}
\nc{\onto}{\twoheadrightarrow}
\nc{\contains}{\supset}
\nc{\transverse}{\pitchfork}
\nc{\uncirc}{\underline{\circ}}
\nc{\Jbar}{\overline{J}}
\nc{\Fbar}{\overline{F}}
\nc{\delbar}{\overline{\del}}
\nc{\thetabar}{\overline{\theta}}
\nc{\omegabar}{\overline{\omega}}

\nc{\colldiff}{\diff^{\del}} 
\nc{\trbar}{\overline{T^*\RR}}
\nc{\tr}{T^*\RR}
\nc{\tsa}{Ts\cA}
\nc{\tsb}{Ts\cB}
\nc{\cmbar}{\overline{\cM}}
\nc{\crbar}{\overline{\cR}}

\nc{\vece}{ {\vec \epsilon}}	
\nc{\vecd}{ {\vec \delta}}
\nc{\ov}{\overline}
\DMO{\op}{op}
\nc{\opp}{ ^{\op}}

\nc{\hiro}{\textcolor{blue}}

\nc{\eqn}{\begin{equation}}
\nc{\eqnn}{\begin{equation}\nonumber}
\nc{\eqnd}{\end{equation}}
\nc{\enum}{\begin{enumerate}}
\nc{\enumd}{\end{enumerate}}

\def\cA{\mathcal A}\def\cB{\mathcal B}\def\cC{\mathcal C}
\def\cG{\mathcal G}
\def\cL{\mathcal L}
\def\cM{\mathcal M}
\def\cR{\mathcal R}
\def\cW{\mathcal W}

\def\RR{\mathbb R}

\def\ZZ{\mathbb Z}

\def\sD{\mathsf D}

\title{The Fukaya category pairs with Lagrangian cobordisms exactly}
\author{Hiro Lee Tanaka}
\begin{document}

\maketitle

\begin{abstract}
We prove that a pairing between the Fukaya category and the $\infty$-category of Lagrangian cobordisms respects mapping cones. This is another step toward constructing a lift of Fukaya categories to the level of spectra (in the sense of stable homotopy theory). As corollaries, we show that the map in~\cite{tanaka-pairing} from cobordism groups to Floer cohomology lifts to the level of spectra, and one also recovers some results of Biran and Cornea for what we call ``vertically collared'' cobordisms.
\end{abstract}

\tableofcontents

\section{Introduction}
Fix $M$ a suitably convex symplectic manifold. 
In~\cite{tanaka-pairing}, we constructed a functor
	\eqnn
		\Xi: \lag_\Lambda(M) \to \wrap_\Lambda(M)\Mod.		
	\eqnd
The domain is an $\infty$-category build of Lagrangian branes and their cobordisms. The target is the $\infty$-category of modules over a Fukaya category. The main result of the present work is:

\begin{theorem}\label{theorem}
$\Xi$ is an exact functor. That is, it respects zero objects and sends fiber sequences to fiber sequences.
\end{theorem}

We briefly recall the $\infty$-categories involved in this functor.

The domain of the functor, $\lag_\Lambda(M)$, is an $\infty$-category whose objects are branes in $M$, or in $M \times T^*E^n$ for some Euclidean space $E^n$. Its morphisms are Lagrangian cobordisms between objects, and its higher cobordisms are higher-dimensional Lagrangian cobordisms. Importantly, one has the freedom to choose a subset $\Lambda \subset M$ with respect to which all cobordisms of $\lag_\Lambda(M)$ must be non-characteristic.  The functor $\Xi$ applies to the case where $\Lambda$ equals the skeleton of $M$, or the case where $\Lambda$ equals $M$ itself.\footnote{
It is anticipated that $\Xi$ exists for other choices of $\Lambda$, when $\wrap_\Lambda$ would represent a {\em partially wrapped} Fukaya category. } In either case, $\lag_\Lambda(M)$ is an $\infty$-category which does not use any Floer theory in its definition.

Now we explain the target. When $\Lambda$ is the skeleton of $M$, $\wrap_\Lambda(M)$ is the wrapped Fukaya category of $M$. When $\Lambda$ is $M$ itself, it is the full subcategory of those objects which are compact Lagrangian submanifolds of $M$. The notation $\wrap\Mod$ denotes the $\infty$-category of contravariant $A_\infty$-functors from $\wrap_\Lambda(M)$ to cochain complexes. When $Y \subset M$ is a brane in one of these Fukaya categories, it is also an object of $\lag_\Lambda(M)$, and $\Xi(Y)$ is the module represented by $Y$.

One consequence of the existence of $\Xi$ is that certain Lagrangian cobordisms induce equivalences in the Fukaya category; moreover, homotopy groups of spaces of Lagrangian cobordisms are detected in Floer cohomology groups. This paper develops our study of $\Xi$ through the proof of its exactness, in particular showing that these group maps arise as maps of homotopy groups of spectra.

What do we mean by this? We proved in~\cite{nadler-tanaka} that $\lag_\Lambda(M)$ is a stable $\infty$-category for any $\Lambda$.\footnote{One can think of {\em stability} for an $\infty$-category as a generalization of being {\em pretriangulated} for a dg- or $A_\infty$-category. For example, the homotopy category of any stable $\infty$-category is triangulated. The generalization allows one to deal with $\infty$-categories that are not necessarily linear over $\ZZ$.} As a consequence, it has mapping cones, direct sums, and shift functors. Moreover, the recipient of $\Xi$---the category of $A_\infty$-modules over $\wrap$ or $\wrap^\compact$---is also a stable $\infty$-category because it is (the nerve of) a pretriangulated dg-category. Hence it is natural to ask whether $\Xi$ respects the stability of both the domain and codomain. Our main result states that it does.

\begin{remark}
Note that the stability of each $\infty$-category arises in different ways. For $\wrap\Mod$, stability is a formal consequence of $\chain$ being a stable $\infty$-category. Put another way, the Fukaya category is formally stabilized by exploiting the fact that it is an $A_\infty$-category (i.e., enriched over chain complexes, which already form a stable $\infty$-category).\footnote{We mean ``enrichment'' in the sense of $\infty$-categories, as developed by~\cite{gepner-haugseng}.} As an illustration of how formal the construction is: Given a degree 0, closed morphism in the Fukaya category, it is not always possible to give an easy geometric interpretation of the mapping cone of this morphism.\footnote{A notable example is when the morphism is a unique intersection point, in which case Lagrangian surgery represents a mapping cone.}

On the other hand, for $\lag$, stability is a {\em geometric} consequence; one does not formally enlarge the geometric objects through algebraic means. This is similar to the way that one naturally finds rich structure on Thom spaces---they allow for suspension-loop maps which make them into infinite loop spaces. One category-level higher from classical Thom spaces, the geometry involved in the $\infty$-category of Lagrangian cobordisms naturally implies that $\lag$ is stable. To contrast this situation with the example of the previous paragraph: Given any cobordism, there is a natural, geometrically defined object which is the mapping cone of that cobordism. (See Section~\ref{section.lag-kernels}.)
\end{remark}

\subsection{Applications}
We present corollaries in increasing order of geometric application. The first follows from general nonsense about exact functors:

\begin{cor}
$\Xi$ preserves all finite limits and colimits.
\end{cor}

Out of an $\infty$-category $\cC$, one can take $\pi_0$ of all its $\hom$ spaces to construct the homotopy category $ho \cC$.\footnote{If the stable $\infty$-category is linear over $\ZZ$, $ho \cC$ is the same as the ``0th cohomology'' category, often denoted $H^0 \cC$. } It is an ordinary category, and when $\cC$ is stable, $ho \cC$ can be given a triangulated structure as follows: One declares a triangle to be distinguished if a lift in the original $\infty$-category comes from a fiber sequence. Since fiber sequences are respected by $\Xi$, we have:

\begin{cor}
The induced map of homotopy categories
	\eqnn
		ho\Xi: ho\lag_\Lambda(M) \to ho\wrap_\Lambda(M)\Mod
	\eqnd
respects distinguished triangles. (Here, $\Lambda$ equals either $\sk(M)$ or $M$.)
\end{cor}

Roughly speaking, we know that the coproduct of two branes in the Fukaya category ought to be their disjoint union. But not every pair of branes in $M$ admits a disjoint embedding in $M$---for instance, if the branes are compact, a necessary condition to admitting a disjoint embedding is that their Floer complex be equivalent to zero. 
However, in $\lag_\Lambda(M)$, the coproduct of branes {\em is} given by their disjoint union: By virtue of stabilizing, branes that could not be disjoint in $M$ can easily be made disjoint in $M \times T^*E^n$.\footnote{Note here that the hom between two disjoint branes in $\lag$ is {\em not} typically zero because of the global nature of cobordisms---one can find cobordisms between far-apart branes. This is one sense in which $\lag$ is similar to $\wrap$: Wrapping also permits morphisms between far-away, non-compact objects. } The next corollary shows that disjoint union of branes---which cannot always be realized in an unstabilized setting---indeed plays the role of coproduct in Fukaya categories.

\begin{cor}
Disjoint unions of stabilized branes induce direct sums of modules over the Fukaya category.
\end{cor}

\begin{proof}
In $\lag$, the mapping cone of a zero map $Y_0[-1] \to Y_1$ is equivalent to the disjoint union $Y_0\times E^\vee \coprod Y_1 \times E^\vee$. On the other hand, the cone of this zero map is always the coproduct of $Y_0$ with $Y_1$ by categorical nonsense. Now note that $\Xi$ preserves coproducts, and that $Y_i \times E^\vee \simeq Y_i$ in $\lag$.
\end{proof}

The next corollary gives another reason why Lagrangian cobordisms---compact or not---detect mapping cones in the Fukaya category. One should, of course, compare this to~\cite{biran-cornea, biran-cornea-2}.

\begin{cor}
Since every Lagrangian cobordism (i.e., every morphism) admits a mapping cone, Lagrangian cobordisms induce mapping cone sequences in the Fukaya category. 
\end{cor}

Concretely, given a cobordism $f: L_0 \to L_1$, the kernel/cone construction from~\cite{nadler-tanaka} defines a new brane, $\cone(f) \subset M \times T^*E$, and a fiber sequence
	\eqnn
		\xymatrix{
			L_0 \ar[r]^f \ar[d]  \ar[dr] & L_1 \ar[d] \\
			0 \ar[r] & \cone(f)
		}
	\eqnd
which, by exactness, is sent to a fiber sequence in $\fukaya\Mod$. When $L_0$ and $L_1$ happen to be honest objects of $\fukaya$, this shows that any cobordism between $L_0$ and $L_1$ induces a fiber sequence in $\fukaya\Mod$. Now, as it turns out, one can prove that if $f$ has only a single vertical end collared by $L'$, then $\cone(f)$ is equivalent to $L'$. Thus, the above fiber sequence in $\lag$ is sent via $\Xi$ to a fiber sequence $L_0 \to L_1 \to L'$ in $\fukaya\Mod$ as well. This recovers a result of~\cite{biran-cornea} at the level of triangulated categories. We emphasize that one did not see any homotopy fiber squares, nor their compatibilities, in~\cite{biran-cornea, biran-cornea-2}---this prevents one from constructing the spectra, and the spectrum-level lifts, that we witness in Corollary~\ref{cor.spectrum-lift}.

\begin{remark}
However, not every cobordism admits vertically collared ends.\footnote{Note that this is another point of departure from~\cite{biran-cornea, biran-cornea-2}, who only consider cobordisms with collared ends.} This is because being eventually conical in $M \times T^*\RR$ does not imply that the brane splits as a product $L \times l$ for some curve $l \subset T^*\RR$. An example is given by a linear-at-infinity Hamiltonian isotopy from $L_0$ to $L_1$---the associated cobordism to this isotopy is not vertically collared by some $l$. Note that if such a Hamiltonian is bounded near the skeleton of $M$, each such cobordism is an equivalence, so their cones are zero objects in $\lag$.
\end{remark}

\begin{remark}
While~\cite{biran-cornea} observed that Lagrangian cobordisms give rise to exact triangles in the Fukaya category a posteriori by applying Floer theory to Lagrangian cobordisms, the work in~\cite{nadler-tanaka} shows that the theory of Lagrangian cobordisms a priori has a stable structure (e.g., admitting mapping cones) without any reference to Floer theory. One outcome of Theorem~\ref{theorem} is the demonstration of a compatibility between the inherent stability of Lagrangian cobordisms and the geometric constructions observed in~\cite{biran-cornea}.
\end{remark}

Finally, recall that in a stable $\infty$-category $\cC$, all $\hom$-spaces naturally inherit the structure of an infinite loop space. To see why, fix two objects $X$ and $Y$. Since $\cC$ has a zero object, $\hom(X,Y)$ has a basepoint given by a composition $X \to 0 \to Y$. By the universal property of the pullback diagram
	\eqnn
		\xymatrix{
			Y \ar[r] \ar[d]  \ar[dr] & 0 \ar[d] \\
			0 \ar[r] & Y[1]
		}
	\eqnd
we see that
	\eqnn
		\ldots \simeq \hom(X,Y) \simeq \Omega(X, Y[1]) \simeq \Omega^2 (X, Y[2]) \simeq \ldots.
	\eqnd
In fact, this sequence shows that $\hom(X,Y)$ is the 0th space of a (possibly non-connective) $\Omega$-spectrum.  Composition is compatible with the loop space structures, so one can justifiably think of a stable $\infty$-category as an $\infty$-category enriched in spectra. Since the loop space structure is induced by fiber sequences, an exact functor induces a map of spectrum-enriched $\infty$-categories.\footnote{While we have kept our discussion of enriched $\infty$-categories heuristic, the interested reader should consult~\cite{gepner-haugseng}.}

\begin{cor}\label{cor.spectrum-lift}
Let $Y_0, Y_1$ be objects of $\wrap$. Let $\hom_\lag(Y_0,Y_1)$ be the spectrum of $\Lambda$-non-characteristic Lagrangian cobordisms between them. Also let $WF^*(Y_0,Y_1)$ be the wrapped cochain complex, which we think of as a $\ZZ$-linear spectrum. Then $\Xi$ induces a map of spectra
	\eqn\label{eqn.hom-map}
		\hom_\lag(Y_0,Y_1) \to WF^*(Y_0,Y_1).
	\eqnd
In particular, one has induced maps of abelian groups
	\eqn\label{eqn.pi-hom-map}
		\pi_k(\hom_\lag(Y_0,Y_1)) \to HF^{-k}(Y_0,Y_1).
	\eqnd
(The negative indexing appears because, by the Dold-Kan correspondence for cochain complexes, negative cohomology groups of a complex are the positive homotopy groups of the associated infinite loop space.) 

In the case that $Y_0 = Y_1$, the induced map of spectra (\ref{eqn.hom-map}) is a map of $A_\infty$ ring spectra by the functoriality of $\Xi$. In particular, the induced map on homotopy groups (\ref{eqn.pi-hom-map}) is a map of associative, unital, graded rings.
\end{cor}

\begin{remark}
Of course, (\ref{eqn.pi-hom-map}) was already known to be a map of abelian groups for $k>0$ by virtue of $\Xi$ being a functor of $\infty$-categories. The real meat of the corollary is in lifting these abelian group maps to a map between spectra, and lifting the graded ring maps to maps of $A_\infty$ ring spectra.
\end{remark}

\begin{example}
We give some examples to indicate some ways in which this homotopy theory of Lagrangian cobordisms departs from classical Pontrjagin-Thom-type cobordism theories.

Let $L_0 = L_1 = \emptyset$ be the empty Lagrangian. Because $\emptyset \in \ob \lag$ is a zero object, its endomorphism spectrum is contractible regardless of $M$. In contrast, classically, the space of cobordisms from $\emptyset$ to itself has homotopy groups equivalent to cobordism classes of compact manifolds. Heuristically, because we allow for certain non-compact cobordisms, these cobordism groups are zero in our setting. See Example~\ref{example.zero} for how non-compact cobordisms induce a zero morphism.

On the other hand, let $L_0 = L_1 = pt$ for $M = \Lambda = pt$. \footnote{We decorate the $pt$ with whatever grading one likes to make it into an object of $\lag(pt)$.} Classically, the space of cobordisms is an infinite loop space, so any notion of ``space of cobordisms from $pt$ to itself'' is homotopy equivalent to the space of cobordisms from $\emptyset$ to itself. However, $\lag_{pt}(pt)$ is not a space, but an $\infty$-{\em category}. Hence $\End(pt)$ need not be homotopy equivalent to $\End(\emptyset)$. Indeed, the former has a non-trivial map to the spectrum $H\ZZ$ (see Corollary~\ref{cor.Z}) while the latter is, as remarked above, the trivial spectrum.
\end{example}

\begin{remark}\label{remark.conjecture}
Finally, recall from~\cite{nadler-tanaka} and~\cite{tanaka-pairing} the conjecture that 
	\eqnn
	\lag_{\sk(M)}(M) \tensor_{\cL} \ZZ \simeq \wrap(M).
	\eqnd
To even define the notion of tensoring a $\cL$-linear $\infty$-category to a $\ZZ$-linear one, one needs the data of an $E_\infty$ ring map from $\cL$ to $H\ZZ$. One outcome of the present paper is the construction of the map of spectra $\cL \to H\ZZ$, which is realized by the case of $M=\Lambda=pt$ and defining $\cL := \hom(pt,pt)$. (The case $M=pt$ is the only connected, non-empty case in which a symplectic manifold is an object in its own Fukaya category.) In~\cite{tanaka-symm}, we will lift this map to a map of $E_\infty$ ring spectra\footnote{Since $H\ZZ$ is an easy spectrum to deal with, the only content in such a statement is that the domain admits an $E_\infty$ structure.}. For now, one can prove the following:
\end{remark}

\begin{cor}\label{cor.Z}
There is an $A_\infty$ map of ring spectra
	\eqnn
		\cL \to H\ZZ
	\eqnd
inducing a surjection on $\pi_0 H \ZZ \cong \ZZ$. 
\end{cor}

\begin{proof}
As in Remark~\ref{remark.conjecture}, we define $\cL$ to be the endomorphism spectrum 
	\eqnn
	\cL := \hom_{\lag_{pt}(pt)}(pt, pt).
	\eqnd
It is a simple exercise to see that, when the point is equipped with a grading and a relative Pin structure, the (wrapped) Fukaya category of $M = pt$ has a triangulated envelope given by finitely generated complexes of $\ZZ$-modules. And since the Lagrangian $pt$ has a single transverse self-intersection, its endomorphism ring is easily computed to be $\ZZ$ itself. So $\Xi$ induces a map of ring spectra
	\eqnn
		\cL \to \hom_{\fukaya(pt)\Mod}(\Xi(pt),\Xi(pt)) \simeq \hom_{\chain_\ZZ}(\ZZ,\ZZ) \simeq H\ZZ.
	\eqnd
This is obviously a surjection on $\pi_0$ because the identity morphism is sent to the identity morphism.
\end{proof}

\begin{remark}
In fact, the map $\pi_0\End_{\lag}(L) \to \pi_0\End_{\fukaya\Mod}(L)$ is always a surjection on the subgroup $\ZZ \subset \pi_0 \End_{\fukaya\Mod}(L)$ generated by the identity.
\end{remark}

\begin{remark}
There is a potential conflict of terminology. The morphisms $\hom_{\lag_\Lambda(M)}(Y_0,Y_1)$ form a spectrum, and represent an invariant of the pair $(Y_0,Y_1)$. Likewise, formal constructions such as Hochschild homology and cohomology are also spectra, and they are an invariant of $M$. These could be called {\em spectral} invariants. In the literature, there are already Floer-type ``spectral'' invariants with a longer history, where the term ``spectral'' is in the sense of eigenvalues of an operator (hence also in the sense of commutative algebra), rather than a spectrum in the sense of stable homotopy theory. These two instances of the word ``spectrum'' are almost always unrelated as far as we know. This conflict does not arise here as we never use the term ``spectral invariants,'' and we will try to avoid this conflict in later writings.
\end{remark}

\subsection{Outline of the proof}
We assume here that the reader is comfortable with the content of Section~\ref{section.recollections}. And in what follows, by a {\em diagram}, we simply mean a functor $D \to \cC$ where $\cC$ is some $\infty$-category, and $D$ is typically the 2-simplex $\Delta^2$ or a square $\Delta^1 \times \Delta^1$. By an equivalence of diagrams, we mean a functor $D \times \Delta^1 \to \cC$ where all arrows $- \times \Delta^1$ are equivalences. 

First, we prove in Lemma~\ref{lemma.zero} that $\Xi$ preserves zero objects. The meat is in proving that fiber squares are preserved.

Given any cobordism $Y: L_0 \to L_1$, we constructed in~\cite{nadler-tanaka} a fiber square $\Delta^1 \times \Delta^1 \to \lag_\Lambda(M)$, which we draw as follows:
	\eqnn
	\xymatrix{
		\ker(Y) \ar[r] \ar[d] \ar[dr]^{T_1}_{T_2} & L_0 \ar[d]^Y \\
		0 \ar[r] & L_1.
	}
	\eqnd
This contains the additional data of an object $\ker(Y)$, a morphism $\ker(Y) \to L_0$, and two higher cobordisms $T_i$---one should think of each $T_i$ as representing a homotopy-commutative triangle. To prove Theorem~\ref{theorem} we must show that, applying $\Xi$ to this square, we obtain a fiber square in $\wrap_\Lambda(M)\Mod$. Informally, this amounts to showing an equivalence $\Xi \circ \ker \simeq \ker \circ \Xi$ compatible with the homotopies $T_i$.

Recall from~\cite{tanaka-pairing} that the module
	\eqnn
	K := \Xi(\ker(Y))
	\eqnd
is defined as follows: One takes the object $\ker(Y)$ as depicted in Figure~\ref{figure.Y-BY}(b) and pairs it against a brane of the form $X \times \beta$. Here $X\in \ob \wrap$ is a test brane, and $\beta$ should be viewed as a small Hamiltonian perturbation of the real line $E \subset T^*E$.\footnote{Here, $X$ may be a brane in $M$, or a brane in the stabilized version, $M \times T^*E^n$ for some $n$.} Immediately, we see that we have little control over this module---if $Y$ is a cobordism with arbitrary behavior near the zero section of $T^*E$, we have no hope of comparing it to a known module.

To remedy this, one fixes a Hamiltonian perturbation from $\beta$ to $\gamma$---here, $\gamma \subset T^*\RR$ is a curve used to define $\Xi$ on morphisms, and one should think of it as having a very negative $E^\vee$ coordinate. Really, one needs only perturb $\beta$ to a curve which is equal to $\gamma$ on $T^*I$ for some compact interval $I \subset \RR$, so this can be achieved using a compactly supported Hamiltonian perturbation. This isotopy allows one to construct a natural equivalence between the module defined by $- \times \beta$ and the module defined by pairing against $- \times \gamma$. 

\begin{remark}\label{remark.directionality}
Importantly, we model this natural equivalence {\em not} by the usual method of continuation maps---rather, we use a pairing with the so-called suspension of the isotopy, which is a Lagrangian cobordism modeling the isotopy.\footnote{Here, too, is a terminology in conflict with terminology from (stable) homotopy theory. This in no way models the suspension space $\Sigma Y$ for any space $Y$.} This should be viewed as equivalent to the continuation map\footnote{See~\cite{oh-unwrapped-continuation} and~\cite{oh-floer-continuity}, where Oh also proposes the idea that all Hamiltonian continuation maps can be expressed through Lagrangian suspensions.}, but the cobordism construction gives us far greater control on the algebra of the map, which has a directionality property---it always looks like a lower-triangular matrix. This underlies most of the statements we are about to outline. Details can be found near the equation~(\ref{eqn.suspension-complex}) and in Section~\ref{section.hamiltonian}.
\end{remark}

\begin{remark}
For lack of a better word, we will still call the pseudoholomorphic strips that we count using the suspension {\em continuation strips}.
\end{remark}

Writing down the natural equivalence of modules, the isotopy induces two equivalences---one for each $T_i$. When applied to the higher cobordism $T_1$, one obtains an equivalence of diagrams
	\eqn\label{eqn.T1-equivalence-1}
		\xymatrix{
			K \ar[r]^{\Xi (k)} \ar[dr]^{\Xi(T_1)} & \Xi (L_0) \ar[d]^{\Xi (Y)} \\
			& \Xi (L_1)
		}
		\qquad
		\xra[\text{Lemma~\ref{lemma.T1-equivalence-1}}]{\sim}
		\qquad
		\xymatrix{
			\ker\Xi(Y) \ar[r]^p \ar[dr]_h & \Xi(L_0) \ar[d]^{\Xi (Y)} \\
			& \Xi(L_1).
		}
	\eqnd
where the righthand side is the upper-right triangle for the usual kernel diagram in $\wrap\Mod$.

When applied to $T_2$, one obtains an equivalence of diagrams which is not so familiar:
	\eqn\label{eqn.T2-equivalence}
		\xymatrix{
			K\ar[d] \ar[dr]_{\Xi(T_2)} \\
			\Xi(0) \ar[r] & \Xi(L_1). 
		}
		\qquad
		\xra[\text{Lemma~\ref{lemma.T2-equivalence}}]{\sim}
		\qquad
		\xymatrix{
			I \ar[d] \ar[dr]^g_{G}  \\
			0 \ar[r] & \Xi(L_1). 
		}
	\eqnd
We describe the module $I$ in Section~\ref{section.T2}.
As it turns out, one can write down, using purely algebra, an equivalence of triangles
	\eqn\label{eqn.T1-equivalence-2}
		\xymatrix{
			I \ar[r] \ar[dr]^g & \Xi(L_0) \ar[d]^{\Xi(Y)}  \\
			 & \Xi(L_1). 
		}
		\qquad
		\xra[\text{Lemma~\ref{lemma.T1-equivalence-2}}]{\sim}
		\qquad	
		\xymatrix{
			\ker\Xi(Y) \ar[r]^p \ar[dr]_h & \Xi(L_0) \ar[d]^{\Xi (Y)} \\
			& \Xi(L_1).
		}
	\eqnd
where the righthand triangle is the same triangle as in~(\ref{eqn.T1-equivalence-1}).

The point of all this is to first exhibit a square involving $I$, and prove that it is equivalent to the square involving $K$; but we require a proof that the composite of the equivalences $(\ref{eqn.T1-equivalence-2})\circ(\ref{eqn.T1-equivalence-1})$ is compatible with the equivalence~(\ref{eqn.T2-equivalence}). This means we must show that the two rectangles
	\eqnn
		\xymatrix{
    		& K\ar[dr] \ar[dl] & \\
    		\ker \Xi(Y) \ar[dr]_h \ar@{}[rr]^{(\ref{eqn.T2-equivalence})}
&& \Xi(L_1) \ar[dl] \\
    		& \Xi(L_1)
		}
	\eqnd
and
	\eqnn
		\xymatrix{
    		& K\ar[dr] \ar[dl] & \\
    		\ker \Xi(Y) \ar[dr]_h \ar@{}[rr]^-{(\ref{eqn.T1-equivalence-2}) \circ (\ref{eqn.T1-equivalence-1})}
				&& \Xi(L_1) \ar[dl] \\
    		& \Xi(L_1)		}
	\eqnd
are equal. It turns out they are not equal, but a slight modification of the composite prism rectifies this---gluing the diagrams together along this common rectangle, one obtains an equivalence of diagrams
	\eqn\label{eqn.square-equivalence}
		\xymatrix{
    		 K\ar[r] \ar[dr] \ar[d] & \Xi(L_0) \ar[d]  \\
    		 \Xi(0) \ar[r] & \Xi(L_1)
		}
	\qquad
		\xra[\text{Lemma~\ref{lemma.rectangles-glue}}]{\sim}
	\qquad
		\xymatrix{
			I \ar[r] \ar[dr]  \ar[d] & \Xi(L_0) \ar[d] \\
			0 \ar[r] & \Xi(L_1).
		}
	\eqnd

As a final step, we show that the equivalence $\eta: \ker \Xi (Y) \to I$ involved in~(\ref{eqn.T1-equivalence-2}) has the following property: If one pulls back the square involving $I$ to a square involving $\ker \Xi(Y)$, one recovers the usual homotopy pullback diagram in $\wrap\Mod$. This is the content of Lemma~\ref{lemma.I-pullback}. In other words, the square involving $I$ is indeed a pullback square, and the theorem is proven.

To conclude this outline, we make two cautionary remarks. First, while we have written this outline using diagrams of modules, we do not give detailed proofs that $I$ is indeed a module, nor that the maps to/from $I$ are indeed module maps. Instead, we actually run all of our arguments after applying an arbitrary test object $X$ to the putative module $I$. So, strictly speaking, what we prove is that when $X$ is evaluated on a square diagram of modules, one obtains a pullback diagram in cochain complexes. This implies our main result anyway, as limits in functor categories (e.g., module categories) are computed object-wise.

Second: While this outline was presented in the most algebraic manner possible, we emphasize that the main geometric arguments are the ones used in the ``directionality'' mentioned in Remark~\ref{remark.directionality}---this directionality was already used repeatedly in~\cite{tanaka-pairing}. Its validity is based on the boundary-stripping arguments of ibid., which in turn is based on the admissibility of certain ``collared'' Floer perturbation data, whose regularity and associated Gromov compactness were also established in ibid. The basic idea is that, using direct product Floer data, we can reduce questions of pseudoholomorphic strips to honest holomorphic strips in $T^*\RR$, where the open mapping theorem implies that disks can only propagate with ``non-backward'' derivatives near intersection points.

\subsection{Some conventions}
We always use cohomological grading, so differential raise degree by 1. By the commutator, we mean the graded commutator, so
	\eqnn
		[d,f] := df - (-1)^{|x|} fd.
	\eqnd
Also, when we evaluate the module $\Xi(L)$ on a test object $X$, we will often write the resulting cochain complex as $(X,L)$. That is,
	\eqnn
		(X,L) := \Xi(L)(X).
	\eqnd

\subsection{Acknowledgments} This work was conducted while I was supported chiefly by the National Science Foundation under Award No. DMS-1400761. 

\clearpage
\section{Recollections}\label{section.recollections}

\subsection{Stable $\infty$-categories}

Let $\cC$ be an $\infty$-category---that is, a simplicial set satisfying the weak Kan condition. 

\begin{defn}[Pullbacks and pushouts]
Consider a map of simplicial sets $\Delta^1 \times \Delta^1 \to \cC$, pictured as the diagram
	\eqnn
		\xymatrix{
			A \ar[r] \ar[d] \ar[dr] & B \ar[d] \\
			C \ar[r] & D.
		}
	\eqnd
It is called a pullback square if it realizes $A$ as the limit of the diagram $B \to D \leftarrow C$. Dually, this square is called a pushout square if it realizes $D$ as the colimit of the diagram $C \leftarrow A \to B$. 
\end{defn}

\begin{defn}
A pullback square with $C=0$ is called a {\em fiber sequence}, or a {\em homotopy kernel diagram}, or a {\em kernel} diagram. We omit the word ``homotopy'' when it is clear we are speaking of an $\infty$-category.

Likewise, a pushout square with $C=0$ is called a {\em cofiber sequence}, or a {\em mapping cone sequence.}
\end{defn}

\begin{defn}[Initial, terminal, zero objects]
An initial object is a colimit of the empty diagram, while a terminal object is a limit of the empty diagram. We say that an $\infty$-category $\cC$ has a zero object if it has an initial and a terminal object, and if a map from the former to the latter is an equivalence. 
\end{defn}

\begin{remark}
Recall that in an $\infty$-category, there is only one notion of a limit or colimit. In the case that the $\infty$-category is (the nerve of) a usual category as introduced by Eilenberg and Mac Lane~\cite{eilenberg-maclane}, this notion agrees with the classical one. When the $\infty$-category arises from a more homotopical setting (such as a combinatorial model category), a limit in the $\infty$-category is equivalent to a {\em homotopy} limit (e.g., in the model category). So when we say ``pullback,'' ``zero object,'' or any other notion of limit/colimit in this paper, we encourage the model-category-minded reader to insert an implied ``homotopy'' before the limit/colimit. We emphasize that we omit the word ``homotopy'' because, in the framework of $\infty$-categories, there is only one natural notion of limit/colimit anyway.

Here is an illustration of this remark: A pushout in the $\infty$-category of cochain complexes, which we denote $\chain$, is computed using the usual notion of mapping cones for a map, rather than the ``naive'' cokernel (i.e., degree-wise quotient) borrowed from a naive category of cochain complexes. Likewise, a pullback is computed using the shift of the mapping cone.
\end{remark}

\begin{defn}[Stable $\infty$-categories]
$\cC$ is called stable if and only if: 
\enum
	\item $\cC$ has a zero object,
	\item Every morphism extends to a fiber sequence and a cofiber sequence, and
	\item A square is a fiber sequence if and only if it's a cofiber sequence.
\enumd
\end{defn}

\begin{defn} 
For any object $A \in \cC$, the object $A[-1]$ is defined by the homotopy pullback square
	\eqnn
		\xymatrix{
		A[-1] \ar[r] \ar[d]  \ar[dr] & 0 \ar[d] \\
		0 \ar[r] & A
		}
	\eqnd
This defines an endofunctor $[-1]$; its effect on morphisms is determined by the universal property of pullbacks. 
\end{defn}

If $\cC$ is stable, $[-1]$ is an inverse functor to the functor $[1]$, which takes an object $A$ to the cofiber of the map $A \to 0$.

We now review the two relevant examples for this paper.	

\subsection{Chain complexes}
Let $\chain$ be the $\infty$-category of cochain complexes over $\ZZ$.\footnote{Of course, the category of cochain complexes is most naturally a dg-category. We render it an $\infty$-category by taking the dg-nerve~\cite{higher-algebra}, or, if one likes, a hom-wise Dold-Kan correspondence~\cite{dold-symmetric-products, kan-functors}.} Then the shift functor $[-1]$ is what one expects: $A[-1]$ is quasi-isomorphic to the cochain complex obtained from $A$ by shifting the gradings by one:
	\eqnn
		(A[-1])^k = A^{k-1}.
	\eqnd
The differential of this cochain complex is {\em minus} the differential of the original cochain complex $A$.
When we say that
	\eqnn
		\xymatrix{
			A[-1] \ar[r] \ar[d] \ar[dr] & 0 \ar[d] \\
			0 \ar[r] & A
		}
	\eqnd
is a pullback diagram in cochain complexes, all the arrows are the zero morphism, but the two-simplices specify homotopies.\footnote{This choice is important: The null-homotopy from the 0 map to the 0 map, for instance, would not exhibit a pullback diagram.} There are many possible choices for this homotopy, but the space of choices is contractible. We take our standard model for the homotopy to be exhibited by the degree -1 element in $\hom(A[-1],A)$ given by (the shift of) the identity map. This, of course, is also the element realizing the equivalence of $\hom$ cochain complexes:
	\eqnn
		\hom(B, A[-1]) \simeq \hom(B,A)[-1].
	\eqnd
\begin{remark}
We will sometimes omit the diagonal edge from $A[-1] \to A$ in our pullback diagrams, especially in $\chain$. This is because in $\chain$, there is a natural homotopy from $A[-1] \to 0 \to A$ to $A[-1] \xra{0} A$, given by the zero element of the hom cochain complex; this has the effect of being able to ``push'' all higher homotopies to either one of the two triangles in the square diagram.
\end{remark}

More generally, fix a morphism $f: A_0 \to A_1$. Then the homotopy kernel of $f$ is determined by the cochain complex
	\eqnn
		\ker f \cong (A_0 \oplus A_1[-1], d),
		\qquad
		d(x_0, x_1) = (dx_0, f(x_0) - dx_1)
	\eqnd
and the diagram	
	\eqn\label{eqn.chain-kernel}
		\xymatrix{
		\ker f \ar[r]^p \ar[d] \ar[dr]  & A_0 \ar[d]^f \\
		0 \ar[r] & A_1
		}
	\eqnd
which we now make precise. The morphism
	\eqn\label{eqn.p}
		p: \ker f \to A_0
	\eqnd
are the obvious projection. The diagonal arrow is $f \circ p$, and the upper-right triangle is given by the 0 homotopy. In the lower-left triangle, both arrows to/from 0 are the 0 map, and the triangle is specified by the element
	\eqn\label{eqn.R}
		R \in \hom^{-1}(\ker f, A_1),
		\qquad
		R(x_0, x_1) = -x_1.
	\eqnd
By abuse of notation, we will also refer to the diagram~(\ref{eqn.chain-kernel}) as $R$. Note that $R$ indeed specifies a homotopy from $f \circ p$ to the 0 map, as $dR + Rd = 0 - f \circ p$.

Finally, let us review the universal property of the homotopy kernel. Fix a diagram $G: \Delta^1 \times \Delta^1 \to \chain$, which we depict as
	\eqnn
		\xymatrix{
		A' \ar[r]^{p'} \ar[d]  \ar[dr] & A_0 \ar[d]^f \\
		0 \ar[r] & A_1.
		}
	\eqnd
Explicitly, there is some morphism $p': A' \to A_0$, and some morphism $A' \to A_1$, but otherwise all edges are as before. One has the additional data of homotopies between the various morphisms.
Then the universal property of the homotopy kernel says there exists a unique (up to contractible choice) morphism $\alpha$ such that when pulling back (\ref{eqn.chain-kernel}) along $\alpha$,
	\eqnn
		\xymatrix{
			A' \ar[dr]^\alpha \\
			& \ker f \ar[r] \ar[d]  \ar[dr] & A_0 \ar[d] \\
			& 0 \ar[r] & A_1
		}
	\eqnd
the resulting diagram is homotopic to G. The formula for $\alpha$ can be made explicit:
	\eqnn
		\alpha(x)
		:=
		(p'(x) , G (x) ).
	\eqnd
Here, we are abusing notation and writing $G$ for the element of $\hom^{-1}(A', A_1)$ realizing the homotopy between $f \circ p'$ and the 0 morphism.

\begin{remark}\label{remark.module-kernel}
Given an $A_\infty$-category $\cA$, the module dg-category $\cA\Mod$ is also a stable $\infty$-category. How? First, one thinks of $\cA\Mod$ as an $\infty$-category by applying Lurie's dg-nerve construction~\cite{higher-algebra}. One notes that a module category's (co)limits are computed object-wise, and thus concludes that the module category is pre-triangulated; we finish by noting that any pre-triangulated dg category has a stable dg-nerve (for a proof, see~\cite{cohn}).\footnote{Another proof would be to show that $\ZZ$-linear functors from the nerve of $\cA$ to the nerve of $\chain$ are equivalent as an $\infty$-category to the nerve of the dg-category $\cA\Mod$. The former is stable again for general reasons.} As an example, if $f: A_0 \to A_1$ is a morphism in the module category, then for any object $X \in \cA$, one has a natural cochain complex
	\eqn\label{eqn.kernel-X}
		\ker f_X = A_0(X) \oplus A_1(X)[-1],
		\qquad
		d(x_0, x_1) = (dx_0, f_X(x_0) - d x_1).
	\eqnd
The diagram (\ref{eqn.chain-kernel}) defines for each $X$ a homotopy kernel diagram
	\eqn\label{eqn.mod-kernel}
		\xymatrix{
		\ker f_X \ar[r]^{p_X} \ar[d]  \ar[dr] & A_0(X) \ar[d]^{f_X} \\
		0 \ar[r] & A_1(X)
		}		
	\eqnd
hence one has a homotopy kernel diagram in $\cA\Mod$ just as in (\ref{eqn.chain-kernel}). 
\end{remark}

\begin{remark}
We will often talks about ``paths'' to describe homotopies in chain complexes; we make explicit what we mean.

Any cochain complex $A$ can be truncated to a new cochain complex $\tau_{\leq 0} A$ concentrated in non-positive degrees, while leaving the non-positive cohomology groups of $A$ unchanged.\footnote{More appropriately: The truncation and identity functors allow for a natural transformation $\tau_{\leq 0} \to \id$ inducing an isomorphism on non-positive cohomology.} By the Dold-Kan correspondence, any cochain complex $\tau_{\leq 0} A$ concentrated in non-positive degrees can be made into a simplicial abelian group, and in particular, a simplicial set $DK(\tau_{\leq 0} A)$. Of course, simplicial sets are a combinatorial model for topological spaces, and in particular admit notions of homotopy groups and the like. The Dold-Kan functor induces isomorphisms
	\eqnn
		H^{-p}(\tau_{\leq 0} A) \cong \pi_p DK(\tau_{\leq 0} A)
	\eqnd
so that cohomology groups can be faithfully thought of as homotopy groups of a space. Moreover, fix two closed elements $a_0, a_1 \in A^0$. These define vertices of $DK(\tau_{\leq 0} A)$. Since the simplices of $DK$ are built from elements of $\tau_{\leq 0}A$, one can show that an element of $A^{-1}$ realizing a homotopy between $a_0, a_1 \in A^0$ induces an edge in $DK(\tau_{\leq 0} A)$ from $a_0$ to $a_1$. This allows us to unambiguously translate between homotopies (in the sense of cochain complexes) and paths (realized as edges in a simplicial set).
\end{remark}

\subsection{Lagrangian cobordisms}\label{section.lag-cone}

\begin{notation} 
We review some notation from~\cite{tanaka-pairing}:
\begin{itemize}
	\item $E = F = \RR$ is the Euclidean line. We use different letters to denote $E^n$ and $F^N$ because these Euclidean spaces play different roles: $E^n$ is a stabilizing direction where objects live. $F^N$ is the direction in which morphisms propagate.
	\item $E^\vee = T^*_0 E$ is the cotangent fiber of $T^*E$ at $0 \in E$. 
	\item $\lag_\Lambda(M)$ will be abbreviated as $\lag$ with the dependence on $\Lambda$ and $M$ suppressed.
	\item $\beta \subset T^*E$ is a non-compact curve which is closed as a subset, and which is parallel to the zero section outside some compact region. One can think of $\beta$ as a small perturbation of the zero section, though we will take some freedom with this interpretation. The key property is that when one pairs $\beta$ against $E^\vee$, one finds that the Floer cochain complex is the ``correct'' answer: It is equivalent to a free $\ZZ$ module with one generator, concentrated in some degree depending on the grading placed on $\beta$ and $E^\vee$. When we need a collection of such $\beta$, they will be indexed as $\beta_0, \beta_1, \ldots $.
	\item $\gamma \subset T^*F$ is also a non-compact curve. One chooses a collection of $\gamma$ for each (higher) morphism $Y \subset M \times T^*F^N$.  One should think of $\gamma$ as a version of $\beta$ which is ``deep enough'' (i.e., has negative enough $F^\vee$ coordinates) to be compatible with the non-characteristic property of morphisms. 
	\item We will often talk of products $L_0 \times L_1 \subset M_0 \times M_1$. When each $L_i$ is an exact, convex Lagrangian of $M_i$ (meaning that their primitives vanish outside a compact subset, and that $L_i$ is eventually invariant under the Liouville flow), it is not true that $L_0 \times L_1$ is also convex. 
	
Given $L_i \subset M_i$, we choose Hamiltonian vector fields $X_i$ such that
	\eqn\label{eqn.products}
	L_1 \times L_2 \to M_1 \times M_2,
	\qquad
	(x_1,x_2) \mapsto (\Phi_{-f_2(x_2)}^{X_1} (x_1), \Phi_{-f_1(x_1)}^{X_2}(x_2))
	\eqnd
is conical---this embedding satisfies the property that $\theta_1 + \theta_2$ vanishes along 
	\eqnn
	\del(L_1 \times L_2) = \del L_1 \times L_2 \bigcup_{\del L_1 \times \del L_2} L_1 \times \del L_2.
	\eqnd
To see such $X_i$ exist, one simply needs to take $X_i$ to be the Hamiltonian flow of some linear Hamiltonian $H_i: M_i \to \RR$---for instance, if $X_i$ is equal to the Reeb flow along $\del M_i$.

\end{itemize}
\end{notation}

\subsubsection{Objects and morphisms}
Recall that an object of $\lag$ is a collection of data
	\eqnn
	(L, f, \alpha, P)
	\eqnd
where $L \subset M \times T^*E^n$ is a Lagrangian for some $n \geq 0$, and the rest of the data are a primitive, a grading, and a relative Pin structure. By definition, a brane $L \subset M \times T^*E^n$ is considered to be the same object as $L \times E^\vee \subset M \times T^*E^n \times T^*E$.\footnote{This identification is a natural symplectic version of usual constructions in cobordism theory: Given $W \subset E^N$, one can further embed $W$ into $E^{N+1}$. The conormals satisfy $T^*_WE^N \times E^\vee = T^*_W E^{N+1}.$}

Assume without loss of generality that $L_0$ and $L_1$ are both submanifolds of $M \times T^*E^n$ for the same $n \geq 0$.
A morphism $Y$ from $L_0$ to $L_1$ is also a collection of data
	\eqnn
		(Y, f_Y, \alpha_Y, P_Y)
	\eqnd
where now $Y$ is a submanifold of $M \times T^*E^n \times T^*F$. Importantly, each $Y$ must be collared as follows: There exists some $t_0, t_1 \in F $ such that 
	\eqnn
		Y|_{(-\infty,t_0]} = L_0 \times (-\infty, t_0] \subset (M \times T^*E^n) \times F \subset (M \times T^*E^n) \times T^*F
	\eqnd
and
	\eqnn
		Y|_{[t_1,\infty)} = L_1 \times [t_1,\infty) \subset (M \times T^*E^n) \times F \subset (M \times T^*E^n) \times T^*F.
	\eqnd 
Finally, $Y$ must also satisfy a {\em $\Lambda$-non-characteristic} condition, which means that there is some $p_0 \in F^\vee$ so that
	\eqnn
		Y|_{p \leq p_0} = \{(y, q, p) \in (M \times T^*E^n) \times F \times F^\vee \text{ such that $p \leq p_0$} \}
	\eqnd
avoids some tubular neighborhood of $\Lambda \times E^n \times T^*F$. Note that there is no ``$T^*$'' in the $E^n$ factor. This $\Lambda$-non-characteristic condition ensures that $X \times \beta^n \times \gamma$ only has intersections where $Y$ is collared.

An $N$-morphism is the data of a Lagrangian $Y \subset M \times T^*E^n \times T^*F^N$, again with decorations, again with certain collaring conditions, and again with a $\Lambda$-non-characteristic requirement. See~\cite{nadler-tanaka} and~\cite{tanaka-pairing} for details.

We emphasize that while the $N=1$ case represents (possibly non-invertible) morphisms between objects, the $N \geq 2$-morphisms should be thought of as simplices in some space of cobordisms. For instance, a cobordism between $Y$ and $Y'$ should be interpreted as a homotopy between $Y$ and $Y'$.\footnote{This follows a general philosophy: While there are two natural choices of what one means by ``invertibility'' in an $(\infty,\infty)$-category, one of the choices necessitates that fully dualizable morphisms are equivalences. The key feature in $\lag$ is that in $F^N = F_1 \times F_2 \times \ldots F_N$, the $\Lambda$-characteristic condition only limits how cobordisms can behave in the $F_N$ direction, while the cobordisms realize a ``full dualizability'' in all other $F$ directions.}

\begin{example}[Zero]\label{example.zero}
The empty brane $\emptyset \subset M \times T^*E^n$ is an object of $\lag$. It is in fact a zero object~\cite{nadler-tanaka}. Depicted in Figure~\ref{figure.zero} are morphisms $\emptyset \to L$ and $L \to \emptyset$ for any object $L \in \ob \lag$. By virtue of $\emptyset$ being a zero object, any morphism to/from $L$ is homotopic to the depicted morphisms.

Concretely, fix a connected, smooth curve $l \subset T^*F$ satisfying the following conditions:
\enum
	\item There exists some real number $q_0 \in F$ such that $l$ is equal to the zero section wherever $q \leq q_0$.
	\item There exists some real number $q_1$ and some real number $p_1 \in F^\vee$ such that $l$ is equal to the vertical ray $\{(q_1, p \geq p_1)\}$ wherever $q \geq q_1$. 
	\item $l$ admits a primitive $f_l: l \to \RR$ such that $df_l = pdq|_l$ and $f_l = 0$ on the above two regions.
\enumd

Then for any $L \subset M \times T^*E^n$, the Lagrangian $L \times l$ is a brane in $(M \times T^*E^n) \times T^*F$. It is by definition collared by $L$ where $q<<0$ in the $F$ component, while it is collared by the empty manifold where $q >>0$. Thus this is a morphism from $L$ to $\emptyset$. The horizontal mirror image $l'$ of $l$ defines another cobordism $L \times l'$, which is a morphism from $\emptyset$ to $L$. See Figure~\ref{figure.zero}.

In either case, the gradings and Pin structure on $l$ and $l'$ are chosen so that the corresponding structures on $L \times l$ and $L \times l'$ restrict to the standard ones on $L \times \text{zero section}$.
\end{example}

\begin{figure}
		\[
			\xy
			\xyimport(8,8)(0,0){\includegraphics[width=2in]{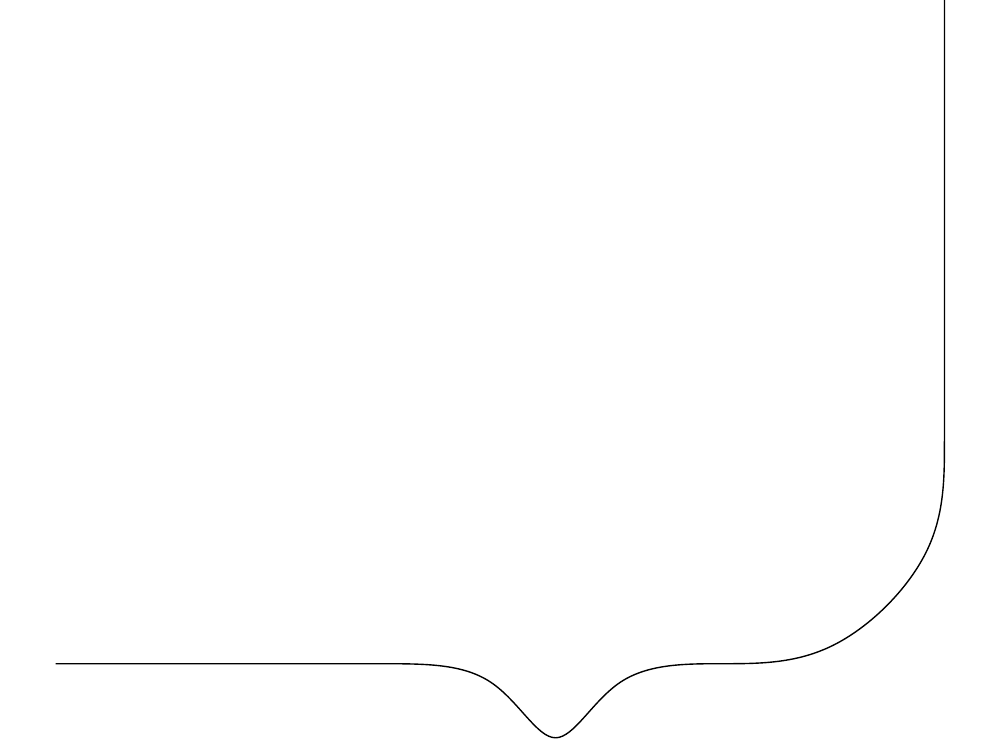}}
			,(4,-1)*+{\text{(a)}}
			\endxy
			\qquad
			\qquad
			\xy
			\xyimport(8,8)(0,0){\includegraphics[width=2in]{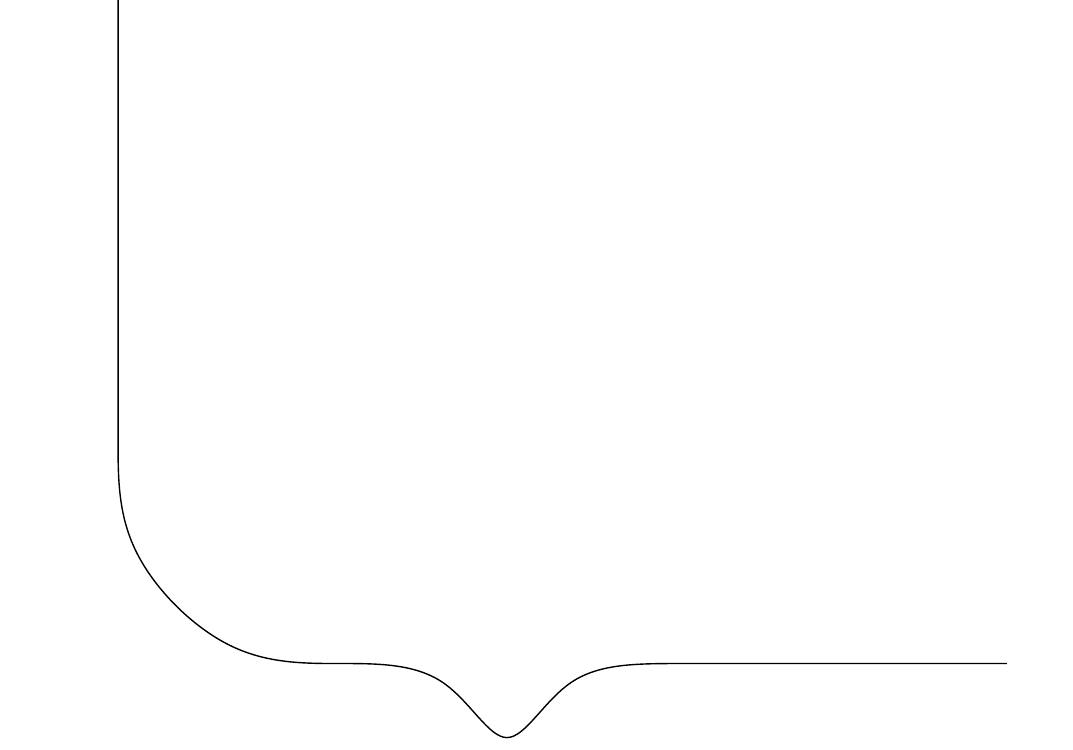}}
			,(4,-1)*+{\text{(b)}}
			\endxy
		\]
\begin{image}\label{figure.zero}
The curves $l$ and $l' \subset T^*F$. (a) The cobordism $L \times l$ is a morphism from $L$ to $\emptyset$. (b) The cobordism $L \times l'$ is a morphism from $\emptyset$ to $L$. Note that $l$ and $l'$ have ``dips'' in them to ensure that one can choose a primitive which equals zero away from a compact subset.
\end{image}
\end{figure}

\subsubsection{Kernels}\label{section.lag-kernels}
To simplify notation, we will assume that all our objects are submanifolds $L \subset M$. The reader can replace $M$ with $M \times T^*E^n$ in what follows to incorporate the cases when $L$ is in a stabilized version of $M$.

Let $Y \subset M \times T^*F$ be a morphism from $L_0$ to $L_1$. Then by definition, its kernel $\ker Y$ is a Lagrangian $B(Y) \subset M \times T^*E$, obtained from $Y$ by removing its infinite horizontal ends and appending a ``cone tail'' instead. See Figure~\ref{figure.Y-BY}. While we wrote this brane as $B(Y)$ in~\cite{tanaka-pairing}, we will hereafter refer to this brane as $\ker Y$ instead.

\begin{figure}
		\[
			\xy
			\xyimport(8,8)(0,0){\includegraphics[height=2in]{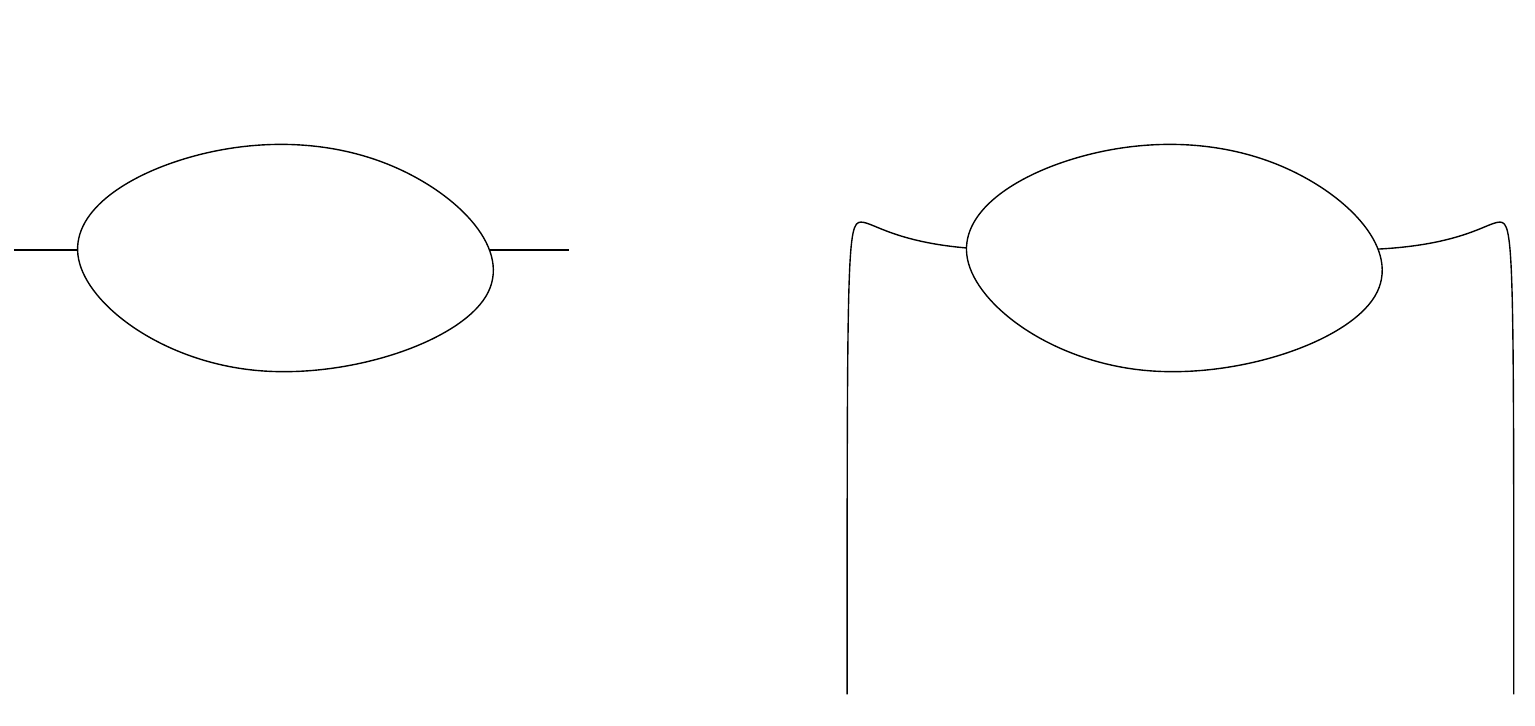}}
			,(1.5,0)*+{\text{(a)}}
			,(6.2,0)*+{\text{(b)}}
			\endxy
		\]
\begin{image}\label{figure.Y-BY}
An image of (a) a cobordism $Y \subset M \times T^*F$, and (b) the kernel object $\ker Y \subset M \times T^*E$, obtained by attaching downward pointing ``tails'' to $Y$. Note that (a) is drawn as the image of the projection of $Y$ to $T^*F$. Likewise, (b) is the image of the projection of $\ker Y$ to $T^*E$. Note also that $Y$ is vertically bounded, in that it is bounded in the $F^\vee$ component, but it need not be so long as it is $\Lambda$-non-characteristic.
\end{image}
\end{figure}

We now review the fiber diagram constructed in~\cite{nadler-tanaka}. We depict it on the left below labeling its arrows, and on the right labeling its 2-simplices:
	\eqn\label{eqn.lag-kernel}
		\xymatrix{
		\ker Y \ar[r]^j \ar[d] \ar[dr]^h & L_0 \ar[d]^Y \\
		0 \ar[r] & L_1,
		}
		\qquad
		\qquad
		\qquad
		\xymatrix{
		\ker Y \ar[r] \ar[d] \ar[dr]^{T_1}_{T_2} & L_0 \ar[d] \\
		0 \ar[r] & L_1.
		}
	\eqnd

Each of these were described in more detail in~\cite{nadler-tanaka}, so we do not rehash the descriptions here. But we do draw the cobordisms giving rise to the simplices:
\begin{itemize}
	\item For $j$, see Figure~\ref{figure.j}.
	\item For $h$, see Figure~\ref{figure.h}.
	\item For $T_1$, see Figure~\ref{figure.T1}. As a cobordism, it is simply the composition $Y^{\dd} \circ j \subset M \times T^*E \times T^*F_2$, thickened by the zero section $F_1 \subset T^*F_1$.
	\item For $T_2$, see Figure~\ref{figure.T2}.
\end{itemize}

To help the reader interpret these images, we give a description of $j$ below. As one can see, it takes a few pages to give a precise description without images; this is why we refer the reader to~\cite{nadler-tanaka} for more detailed descriptions.

\begin{figure}
		\[
			\xy
			\xyimport(8,8)(0,0){\includegraphics[width=5in]{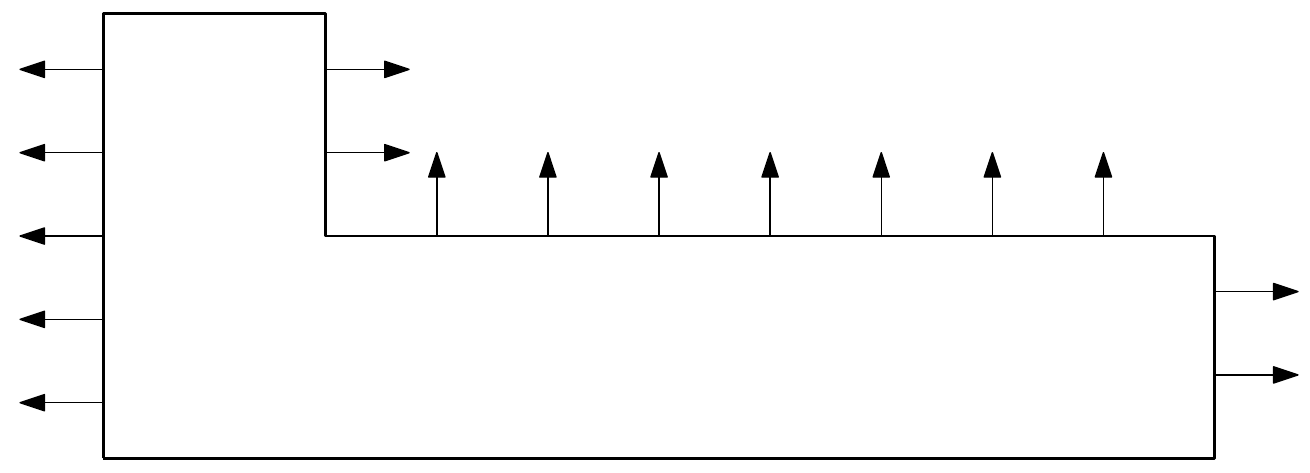}}
			,(4,2)*+{f}
			,(1.3,2)*+{L_0}
			,(1.3,6)*+{L_0}
			,(6.5, 2)*+{L_1}
			\endxy
		\]
\begin{image}\label{figure.j}
The morphism $j$ from $\ker Y$ to $L_0$, depicted by its projection to $E \times F$. The horizontal direction is the $E$ direction (with the rightward direction being positive), while the vertical direction is the $F$ direction (with the upward direction being positive). Note that, by using the standard metric on $\RR^2$, we have depicted the cotangent directions by vector fields. The presence of an arrow signifies that the morphism goes off to the indicated cotangent direction ad infinitum.
\end{image}
\end{figure}

{\bf Construction of $j$.}
Fix $x_0 \in E$ and $q_0 \in F$. Let $K \subset E \times F$ be the subset of those $(x,q)$ such that $x \geq x_0$ and $q \geq q_0$. We fix a smooth function $f$ defined on $K^C$, the complement, such that for some $\epsilon>0$:
	\enum
		\item $f$ is identically zero outside a small $\epsilon$-neighborhood $U(\del K)$ of $\del K$,
		\item $f$ approaches $+\infty$ along $\del K$,  
		\item $f$ does not depend on $x$ when $x > x_0 + \epsilon$,
		\item $f$ does not depend on $q$ when $q > q_0 + \epsilon$, and 
		\item one can isotope the graph Lagrangian $df \subset T^*E \times T^*F$ so that it allows for a primitive equalling zero outside a compact subset, while maintaining the previous two $x$- and $q$-invariant properties.
	\enumd
Let $\Gamma \subset T^*E \times T^*F$ denote this isotoped Lagrangian. Note that by the invariance property, there is a well-defined brane $\Gamma_F \subset T^*F$ collaring $\Gamma$ where $x >>0$.

Now consider $\ker Y$. By translating in the $E$ component if need be, assume that the region where $\ker f$ is collared by $L_0$ is a small neighborhood of $x_0 \in E$. We define $j$ to be the Lagrangian $j \subset M \times T^*E \times T^*F$ given by the union of the following Lagrangians:
\enum
	\item $(\ker Y \times F)|_{U(\del K)^C} \subset M \times T^*E(\del K)^C$.
	\item $(\ker Y)_{x \geq x_0 + \epsilon} \times (\Gamma_F) \subset M \times T^*E_{x \geq x_0 + \epsilon} \times T^*F.$
	\item $(L_0 \times \Gamma)|_{x\leq x_0 + \epsilon, q \geq q_0}$.
\enumd

\begin{remark}
Note that $j$ is collared at $q>>0$ by a brane of the form $L_0 \times w$ where $w$ is some curve in $T^*E$. By construction of $\Gamma$, $w$ is in fact Hamiltonian isotopic to the vertical fiber $E^\vee$ by a Hamiltonian isotopy which is linear at infinity---hence $L_0 \times w$ is equivalent to $L_0 \times E^\vee$ in $\lag$. 
\end{remark}

\begin{remark}
Likewise, the morphism $Y: L_0 \to L_1$ inside $M \times T^*F$ stabilizes to a morphism $Y \times w: L_0 \times w \to L_1 \times w$ inside $M \times T^*E \times T^*F$. 
\end{remark}

\subsection{The functor $\Xi$}
Finally, we briefly recall the definition of $\Xi$. We stop with what $\Xi$ does to 2-simplices, as that's all we need in this paper.

\subsubsection{On vertices}
If $L \subset M \times T^*E^n$ is some brane, $\Xi(L)$ is the module it represents. Concretely, we fix curves $\beta_i \subset T^*E$ for $i=0,1,\ldots$. These are chosen carefully, but one can roughly think of as perturbations of the zero section. Then any brane $X \subset M$ is sent to the cochain complex
	\eqnn
	\Xi(L)(X) := CF^*(X \times \beta_0^n, L).
	\eqnd
Given any finite collection of branes $X_0,\ldots,X_k \in \ob \wrap$, we count holomorphic disks in $M \times T^*E^n$ with boundary on
	\eqnn
		X_0 \times \beta_0^n, \ldots, X_k \times \beta_k^n, L.
	\eqnd
This defines the operations
	\eqnn
		CW^*(X_{k-1}, X_k) \tensor \ldots CF^*(X_0 \times \beta_0^n, L)
		\to
		CF^*(X_k \times \beta_k^n, L)
		\simeq
		CF^*(X_k \times \beta_0^n, L)
	\eqnd
where we have used equivalences
	\eqnn
		CF^*(X_i \times \beta_i^n, X_j \times \beta_j^n) \simeq CW^*(X_i, X_j)
	\eqnd
owing to our choices of $\beta_i$. Note importantly that the stabilization $L^{\dd} = L \times E^\vee \subset M \times T^*E^{n+1}$ represents the same module as $L$ because
	\eqnn
		CF^*(\beta_i, E^\vee) \cong \ZZ.
	\eqnd

\begin{lemma}\label{lemma.zero}
$\Xi$ preserves zero objects.
\end{lemma}

\begin{proof}
Recall from~\cite{nadler-tanaka} that $L=\emptyset$ is a 0 object of $\lag_\Lambda(M)$. Well, the empty brane has no intersection with $X \times \beta^n$ for any test object $X$ and for any $n$; hence $\Xi(L)(X)$ is the 0 cochain complex. In other words, $\Xi$ sends the empty brane to the 0 object of $\wrap\Mod$.
\end{proof}

\subsubsection{On edges}\label{section.1-simplex}
An edge in $\lag_\Lambda(M)$ is a brane $Y \subset M \times T^*E^n \times T^*F$; it must satisfy a collaring condition so that $Y$ eventually equals $F \times L_0$ and $F \times L_1$. Because of this collaring, we can modify $Y$ by adjoining two ``tails'' to $Y$---rather than extending $F \times L_0$ indefinitely, we replace the zero section by a curve in $T^*F$ which approaches negative infinity at either collared end. This is depicted in Figure~\ref{figure.Y-BY}. This new brane is denoted $B(Y) \subset M \times T^*E^n \times T^*F$ in~\cite{tanaka-pairing}, and it is diffeomorphic to $\ker Y \subset M \times T^*E^{n+1}$; the identity bijection $E \cong \RR \cong F$ identifies $B(Y)$ with $\ker Y$, but we treat the two as living in two different manifolds.

Then we must product a morphism of modules $\Xi(L_0) \to \Xi(L_1)$. To define this map, given a test object $X$, consider the Floer cochain complex
	\eqnn
		CF^*(X \times \beta_0^n \times \gamma, B(Y)). 
	\eqnd
By the assumption that $Y$ is $\Lambda$-non-characteristic, one can prove that there is a curve $\gamma$ with negative enough $F^\vee$ coordinate such that---as a  graded abelian group---this Floer complex is isomorphic to
	\eqnn
		CF^*(X \times \beta_0^n, L_0) \oplus 
		CF^*(X \times \beta_0^n, L_1)[-1]
	\eqnd
and the differential is as follows:
	\eqnn
            \left(
            \begin{array}{ccc}
              d_{(X,L_0)} & 0 \\
              \Xi(Y)_X & -d_{(X,L_1)} 
            \end{array}
            \right)
	\eqnd
Here, $d_{(X,L_i)}$ is the differential of the cochain complex $CF^*(X,\beta_0^n, L_i)$. This upper-triangular form is an outcome of the directionality we mentioned in Remark~\ref{remark.directionality}, which we discuss further in Section~\ref{section.directionality} below.

Now, note that the lower-left entry of the matrix, $\Xi(Y)_X$, is necessarily a chain map from $CF^*(X \times \beta^n_0, L_0)$ to $CF^*(X\times \beta_0^n, L_1)$ because the matrix must square to zero. One can prove that this linear map is natural in $X$---it is the first of many operations that defines a map of modules $\Xi(Y): \Xi(L_0) \to \Xi(L_1)$. 

\begin{remark}
Also note that the cochain complex above is not the cochain complex $\Xi(\ker Y)(X)$---the latter is given by
	\eqnn
		CF^*(X \times \beta_0^{n+1}, B(Y)).
	\eqnd
The two are related by choosing a Hamiltonian isotopy which takes the relevant parts of $\gamma$ to $\beta_0$.
\end{remark}

\subsubsection{On 2-simplices}
A 2-simplex in $\lag$ is given by a brane $Y \subset M \times T^*E^n \times T^*F^2$, collared as pictured in Figure~\ref{figure.BY-2-simplex}. The collaring conditions are as follows:
	\enum
		\item $F_1$ is negative, $Y$ is collared by some morphism $Y_{02} : L_0 \to L_2$. 
		\item Where $F_1$ is positive, $Y$ is collared by the composite morphism $Y_{12} \circ Y_{01}$. 
		\item Where $F_2$ is negative, $Y$ is collared by the object $L_0$, and
		\item Where $F_2$ is positive, $Y$ is collared by the object $L_2$.
	\enumd
Then one constructs a new brane $B(Y)$, living in the same manifold as $Y$. Informally, this is constructed in four steps:
	\enum
		\item Rotate $Y_{01}$ so that the $Y_2$ component turns into the $Y_1$ component. 
		\item Since $Y_{01}$ is collared by $L_1$ on one end, we can ``fill out'' the rotated $Y_{01}$ by a copy of $L_1$ times the zero section.
		\item Next, just the copy of $Y_{12}$ by the zero section $F_1$.
		\item Now replace the collared ends of the resulting brane by tails---informally, if $l \subset T^*F$ is the tail used in constructing the cone, one replaces the collared ends by $l^2$.
	\enumd
Now, to construct the 2-simplex in $\wrap\Mod$, we again discuss what we do on a test object $X \subset M$. One considers the Floer cochain complex
	\eqnn
		CF^*(X \times \beta^n \times \gamma^2, B(Y))
	\eqnd
which breaks up as a direct sum
	\eqnn
		CF^*(X \times \beta^n, L_0)
		\oplus
		CF^*(X \times \beta^n, L_1)[-1]
		\oplus
		CF^*(X \times \beta^n, L_2)[-1]
		\oplus
		CF^*(X \times \beta^n, L_2)[-2]
		\oplus
	\eqnd
and whose differential can be written as follows:
	\eqn\label{eqn.2-simplex-differential}
        \left(
        \begin{array}{cccc}
        d_{(X,L_0)}  & 0&0&0\\
        \Xi(Y_{01})  & -d_{(X,L_1})  &  0&0\\
        \Xi(Y_{02})  & 0  &  -d_{(X,L_2)} & 0 \\
        \Xi(Y)  & \Xi(Y_{12})  &  -\id & d_{(X,L_2)}
        \end{array}
        \right)
	\eqnd
This matrix can be drawn as follows, ignoring the diagonal terms, which are implicit:
	\eqnn
		\xymatrix{    		
    		(X,L_2)[-1] \ar[r]^{-\id} & (X,L_2)[-2] \\
    		(X,L_0) \ar[u]^{\Xi(Y_{02})} \ar[r]_{\Xi(Y_{01})} \ar[ur]^{\Xi(Y)}& (X,L_1)[-1] \ar[u]_{\Xi(Y_{12})}
		}
	\eqnd
The notation $(X,L_i)$ is shorthand for $CF^*(X \times \beta^n_0, L_i)$.
Since this differential squares to zero, we obtain the relation
	\eqnn
		[d,\Xi(Y)] = \Xi(Y_{12}) \circ \Xi(Y_{01}) - \Xi(Y_{02}).
	\eqnd
Thus $\Xi(Y)$ is indeed a 2-simplex realizing the homotopy between the composite $\Xi(Y_{12}) \circ \Xi(Y_{01})$ and $\Xi(Y_{02})$.

\begin{figure}
		\[
			\xy
			\xyimport(8,8)(0,0){\includegraphics[width=5in]{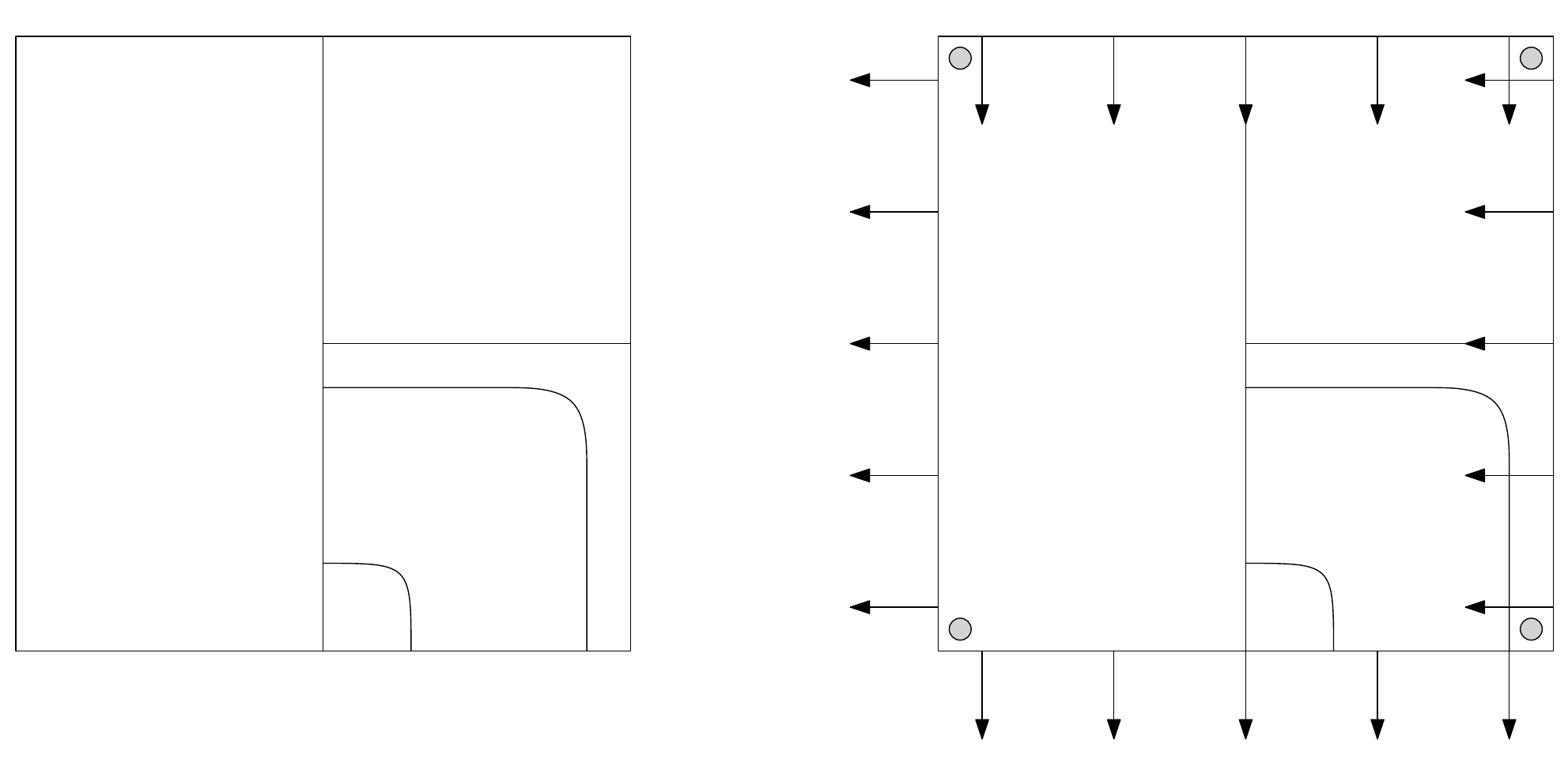}}
			,(0,0.8)*+{L_0}
			,(1.9,1.5)*+{L_0}
			,(3.1,4.1)*+{L_1}
			,(3.3,0.8)*+{L_1}
			,(3.3,8)*+{L_2}
			,(0,8)*+{L_2}
			,(2.3,3)*+{Y_{01} \text{ rotated} }
			,(0.9,4.5)*+{Y}
			,(-0.1,4.5)*+{Y_{01}}
			,(2.4,6)*+{Y_{12} \times F_1}
			,(4.5,0.8)*+{(X,L_0)}
			,(4.7,8)*+{(X,L_2)[-1]}
			,(8.5,0.8)*+{(X,L_0)[-1]}
			,(8.5,8)*+{(X,L_2)[-2]}
			\endxy
			\]
\begin{image}\label{figure.BY-2-simplex}
An image of $B(Y)$ for $Y$ a 2-simplex. The righthand picture depicts the $(F^\vee)^2$ coordinates (as vector fields) and the generators of $CF^*(X \times \beta^n_0 \times \gamma^2, B(Y))$ (labeling the grey dots). Both are drawings in $F^2$, with the $F^\vee$ components omitted. The arrows are dual to the $(F^2)^\vee$ cotangent coordinates, and the grey dots represent where $(X \times \beta_0^n) \times \gamma^2$ intersects $B(Y)$.
\end{image}
\end{figure}

\subsubsection{Directionality}\label{section.directionality}
There are two things we'd like to emphasize about the recollection: The differentials associated to $B(Y)$---whether $Y$ be a 2-simplex or a 1-simplex---have two commonalities:
	\enum
		\item The matrix is upper-triangular, and
		\item The entries of the matrix are often terms we recognize. For instance, in the case $Y$ is a 2-simplex, three of the matrix entries are terms arising from the 1-simplices $Y_{ij}$. And along the diagonal, one recognizes the differentials $d_{(X,L_i)}$. 
	\enumd
Both are consequences of the collaring condition on our branes and our Floer perturbation data. Roughly, in regions near our intersection points, we know that the branes and the perturbation data (Hamiltonians and almost-complex structures) split as a direct product.\footnote{That this direct product structure is regular was proven in~\cite{tanaka-pairing}.}

The power of this is that the projection map to any of the $T^*E$ or $T^*F$ factors is holomorphic---in particular, given any pseudoholomorphic map $u: S \to M \times T^*E^n \times T^*F^N$, the composite with a projection to $T^*E$ or $T^*F$ is an honest holomorphic map in some open subset of $S$.

This allows us to use the open mapping theorem, which in turn allows us to use the boundary-stripping arguments of~\cite{tanaka-pairing}. This allows us to conclude that---in a neighborhood of a Lagrangian intersection point---the directional derivative of $u$ must be non-negative in all $E$ or $F$ components. This is what guarantees property (1): the lower-triangular property, or the ``directionality,'' of the Floer differentials. When the directional derivative is zero, the map $u$ is only non-constant in an orthogonal factor, hence one can recover a differential or $\Xi(Y)$ term computed in lower dimensions. This explains property (2).

Finally, note that the curves $\gamma \subset T^*F$ are chosen to have very negative $F^\vee$ coordinate near $B(Y)$. As a result, when we draw an image of a brane $Q \subset M \times T^*F^2$ with an accompanying vector field, the Floer complex $CF^*(X \times \gamma^2, Q)$ only has generators in regions of $Q$ where the vector field has arbitrarily negative $F_i^\vee$ components. For instance, in Figure~\ref{figure.BY-2-simplex}, the grey intersection dots only appear where the vector fields have large negative values in both $F_i^\vee$ components---i.e., the corners.

\subsection{Hamiltonian isotopies}\label{section.hamiltonian}
When studying compact Lagrangians, we consider two to be equivalent if they are related by a (possibly time-dependent) Hamiltonian isotopy. However, in the context of convex symplectic geometry and convex, non-compact, Lagrangians, one must specify the type of Hamiltonians we consider beyond the compactly-supported type.

We let $H_t: M \to \RR$ be a time-dependent Hamiltonian such that 
$H_t(x) = 0$ for $|t| >>0$.
We define
	\eqnn
	\Phi: \RR \times M \to M,
	\qquad
	(t,x) \mapsto \Phi_t(x)
	\eqnd
where $\Phi_t$ is the result of flowing along the time-dependent Hamiltonian $H_t$ for time $t$. 

\begin{defn}
Fix a brane $L \subset M$. We say that the Hamiltonian is {\em convex for $L$} if the trace (aka the suspension)
	\eqnn
	Q = \{(\Phi_t(x), t, -H_t(\Phi_t(x))) \text{ such that $x \in L$}\} \subset M \times T^*F
	\eqnd
 of the Hamiltonian flow is a convex brane---i.e., $Q$ is exact and its primitive vanishes along $\del Q \subset \del(M \times T^*\RR$). 
\end{defn}

This puts a large restriction on the behavior of $H_t$ near $\del M$. For instance, on $\del Q$, one must have the equation
	\eqn\label{eqn.convex-H}
	\theta_M(X_{H_t}) = H_t.
	\eqnd

\begin{example}
If $M = T^*\RR^n$ with $\theta = \sum_i p_i dq_i$, such an $H$ must only have a {\em linear} dependence on the $p_i$ coordinates near $\del Q$. That is, near the boundary of $M$, and where $L$ flows, $H_t$ must have the form
	\eqn\label{eqn.convex-H-local}
	\sum p_i f_i(q_1,\ldots,q_n, t)
	\eqnd
where $f_i$ are smooth functions with the indicated possible dependencies.
This is a straightforward consequence of the equation (\ref{eqn.convex-H}), which in this context reduces to the series of equations
	\eqnn
	p_i {\frac {\del H} {\del p_i}} = H
	\qquad
	i = 1, \ldots, n
	\eqnd
near $\del M$. If one assumes that $f_i$ are constants, one can use such examples to prove that a conormal at point $q$ is cobordant to the conormal of another point $q' \in \RR^n$; and in fact, equivalent in $\lag$ by noting that these cobordisms avoid $\Lambda$ at $\pm \infty dt$.
\end{example}

Now, the rigidity. The following implies, for instance, that the conormals to non-isotopic submanifolds are not related by flows of convex Hamiltonians. Also, the antidiagonal inside $T^*\RR^{2n}$ is the conormal of a point.

\begin{prop}
If $M = T^*\RR^n$, let $\pi: M \to \RR^n$ be the projection to the zero section. If the boundaries of  two submanifolds $L_0, L_1\subset M$ do note have isotopic projections $\pi(\del L_i)$, then the $L_i$ are not related by a convex Hamiltonian flow.
\end{prop}

\begin{proof}
Let $Q$ be the suspension of a flow by some $H_t$. For $H_t$ to be convex, we know $H_t$ must have the form as indicated in (\ref{eqn.convex-H-local}) near the boundary of $\del Q$. This has associated Hamiltonian vector field
	\eqnn
	X_{H_t}
	=	\sum_i f_i {\frac \del {\del q_i}}
	-
	\sum_{i,j} p_i {\frac {\del f_i} {\del q_j}} {\frac \del {\del p_j}}.
	\eqnd
The flow defined by $X_{H_t}$ only has a dependence on the $q$ coordinates in the $q$ component, hence the following diagram commutes:
	\eqnn
	\xymatrix{
	\RR \times \del(T^*\RR^n) \ar[rr]^{\text{flow along}}_{X_{H_t}} \ar[d]^{\id_\RR \times \pi}
	&& \del(T^*\RR^n) \ar[d]^\pi \\
	\RR \times \RR^n \ar[rr]^{\text{flow along}}_{\sum_i f_i {\frac {\del}{\del q_i}}} 
	&&  \RR^n.
	}
	\eqnd 
Thus if $L_0$ and $L_1$ are related by the flow of $H_t$, the projections of their boundaries are related by a flow in $\RR^n$, hence isotopic.
\end{proof}

Clearly, the above proof patches together to prove the same result for any convex manifold that can be covered by convex manifolds of the form $T^*\RR^n$. In other words, for cotangent bundles.

\begin{corollary}
If $M = T^*B$, let $\pi: M \to B$ be the projection to the zero section. If the boundaries of two submanifolds $L_0, L_1 \subset M$ do note have isotopic projections $\pi(\del L_i)$, then the $L_i$ are not related by a convex Hamiltonian flow.
\end{corollary}

\begin{example}
This is in contrast to the fact that an arbitrary Hamiltonian flow can actually take a conormal to a point and flow it to the zero section---for instance, via the Hamiltonian $H = p^2 + q^2$.
\end{example}

\subsection{Inverse cobordisms}\label{section.inverses}
If $Y \subset M \times T^*F$ is a cobordism which is vertically bounded---i.e., which is bounded in the $F^\vee$ component---then it is an invertible morphism in $\lag_\Lambda(M)$. In particular, $\Xi(Y): \Xi(L_0) \to \Xi(L_1)$ is an invertible map of modules. In fact, one can product an inverse morphism explicitly: The orientation-reversing diffeomorphism $t \mapsto -t$ induces a symplectomorphism $T^*\RR \to T^*\RR$, and the image of $Y$ under this symplectomorphism is called $\overline{Y}$, and this is an inverse to $Y$. We show two ways in which we can see this in Figure~\ref{figure.inverse-cobordism}.

\begin{figure}
		\[
			\xy
			\xyimport(8,8)(0,0){\includegraphics[height=2in]{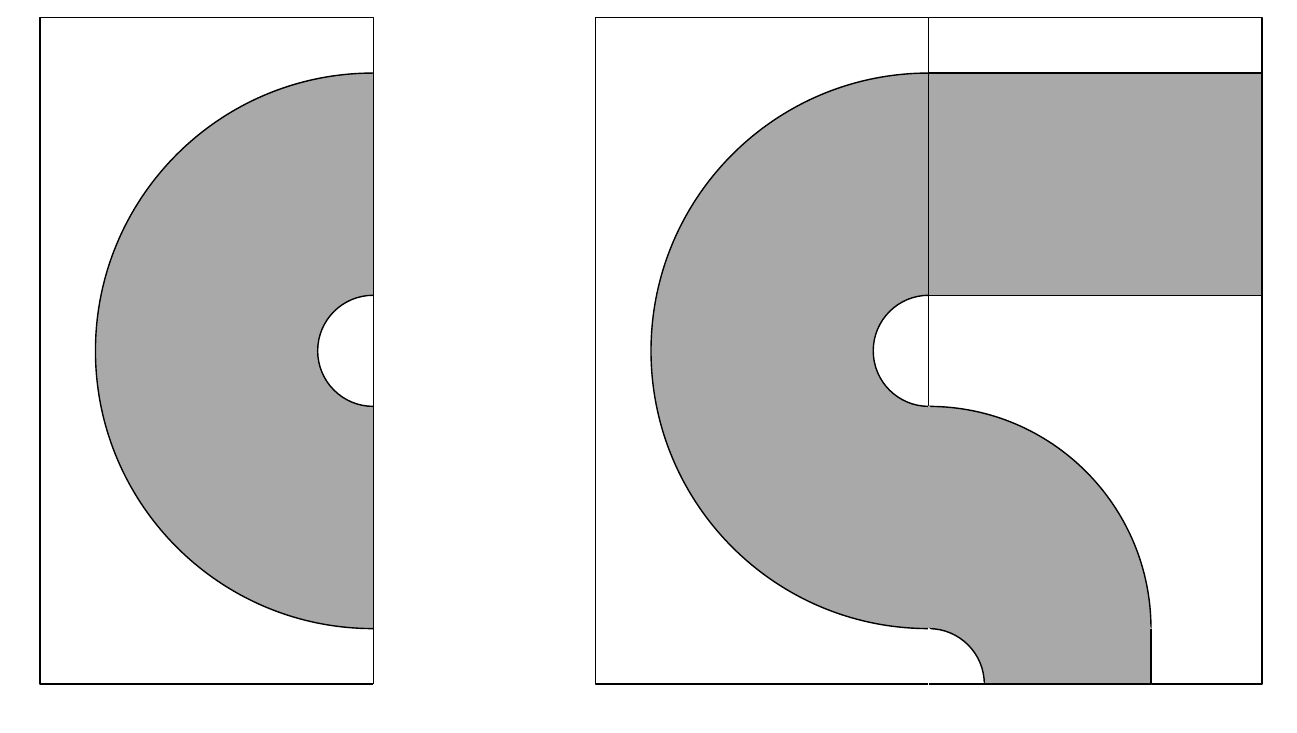}}
			,(1.2,0)*+{(a)}
			,(5.4,0)*+{(b)}
			,(2.4,4.2)*+{L_1}
			,(2.2,2.4)*+{Y}
			,(2.2,6)*+{\overline{Y}}
			,(0.7,1)*+{L_0}
			,(7,4)*+{L_1}
			,(4,0.9)*+{L_0}
			,(5.7,2.4)*+{Y}
			,(5.7,6)*+{\overline{Y}}
			\endxy
			\xy
			\xyimport(8,8)(0,0){\includegraphics[height=2in]{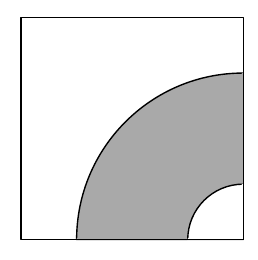}}
			,(4.5,0)*+{(c)}
			,(7,1)*+{L_1}
			,(1.85,6)*+{L_0}
			,(4.5,0.9)*+{Y}
			,(7.3,4)*+{\overline{Y}}
			\endxy
		\]
\begin{image}\label{figure.inverse-cobordism}
(a) depicts a 2-simplex in $\lag$. If $Y$ is a vertically bounded cobordism from $L_0$ to $L_1$, a full rotation (as pictured in the grey region) results in a higher cobordism showing that $\id_{L_0}$ is homotopic to $\overline{Y} \circ Y$, where $\overline{Y}$. (b) Is the result of applying the $B$ construction to this higher cobordism. (c) More directly, in (c) is depicted a brane living over a square; pairing this with $X \times \gamma^2$ shows that $\overline Y \circ Y$ is homotopic to $\id_{L_0}$. The chain map resulting from (c) is homotopic to the one resulting from (b), as these are each clearly isotopic to each other.
\end{image}
\end{figure}

\begin{figure}
		\[
			\xy
			\xyimport(8,8)(0,0){\includegraphics[width=2.5in]{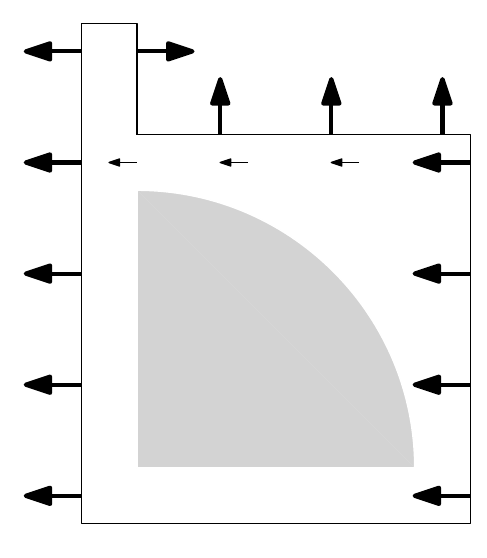}}
			,(1.9,1)*+{L_0}
			,(1.85,7)*+{L_1}
			,(6,5)*+{L_1}
			,(1.8,3)*+{Y}
			,(4.5,0.9)*+{Y}
			,(4.5,3)*+{Y}
			\endxy
		\]
\begin{image}\label{figure.h}
The morphism $h$ from $\ker Y$ to $L_1$. Note that it factors through a brane $L_1 \times b$, where $b$ is drawn by the small horizontal arrows running to the left. The vertical direction is the $F$ direction, in which the morphism propagates. The horizontal direction is the $E$ direction.
\end{image}
\end{figure}

\begin{figure}
		\[
			\xy
			\xyimport(8,8)(0,0){\includegraphics[height=2.7in]{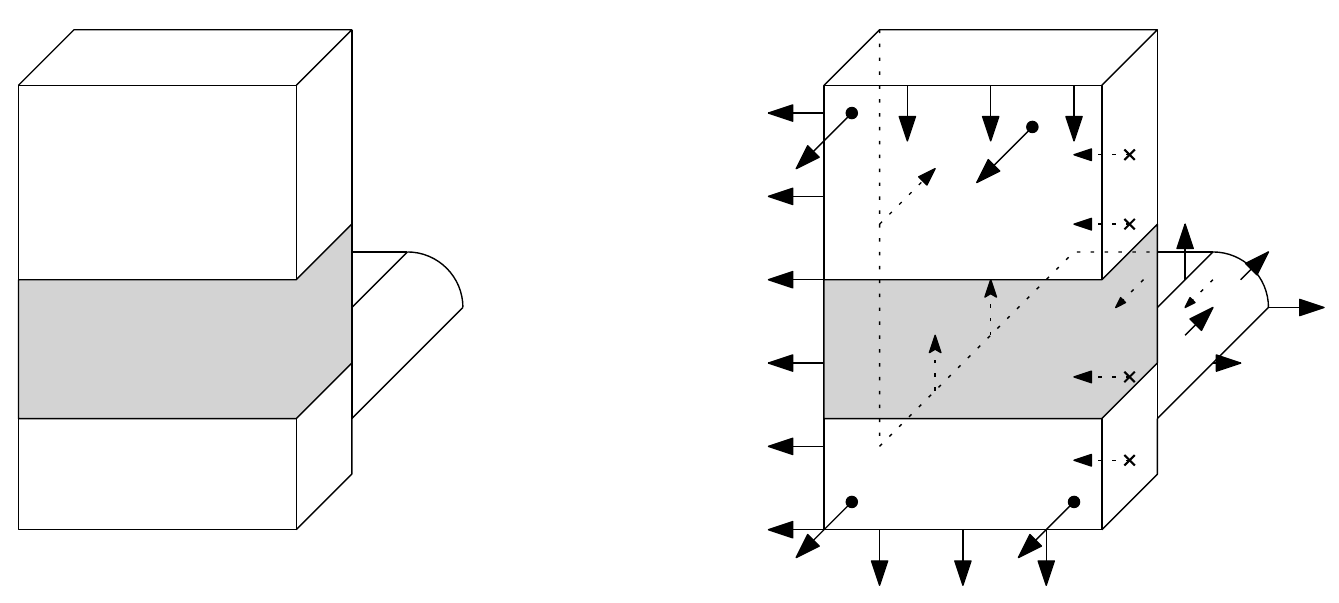}}
			,(1,1.4)*+{L_0}
			,(1,3.5)*+{T}
			,(1,5.5)*+{L_1}
			,(1,-0.5)*+{(a)}
			,(6,-0.5)*+{(b)}
			\endxy
		\]
		\[
			\xy
			\xyimport(8,8)(0,0){\includegraphics[height=2in]{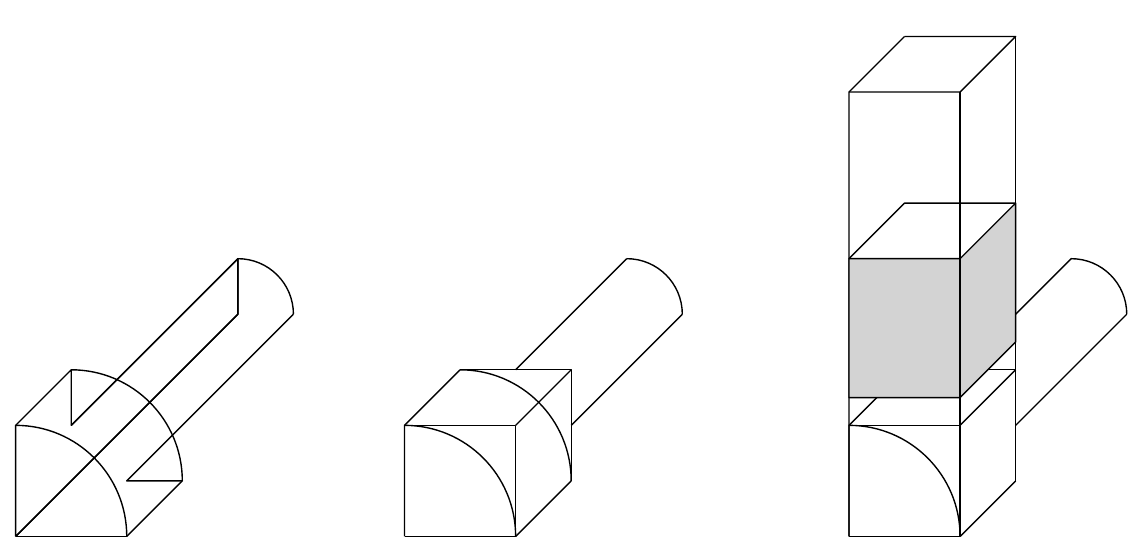}}
			,(0.7, -0.5)*+{(c)}
			,(3.3, -0.5)*+{(d)}
			,(6, -0.5)*+{(e)}
			\endxy
		\]
\begin{image}\label{figure.T1}
At the top (a) is picture of triangle $B(T_1)$, where $T_1$ is the triangle  in~(\ref{eqn.lag-kernel}). The grey region is a copy of the morphism $Y \subset M \times T^*F_2$ with a brane in $T^*(E \times F_1)$. (b) depicts the components of $E^\vee$ and $(F^2)^\vee$ as vector fields. At the bottom is an illustration of how to visualize the $B$ construction of $T_1$ as pictured. (c) is a picture of $j$ being rotated, (d) fills out the rotation with a copy of $L_0$, and (e) attaches a product of $Y^\dd$, as prescribed by the definition of $B$ in~\cite{tanaka-pairing}.
\end{image}
\end{figure}

\begin{figure}
		\[
			\xy
			\xyimport(8,8)(0,0){\includegraphics[width=2.5in]{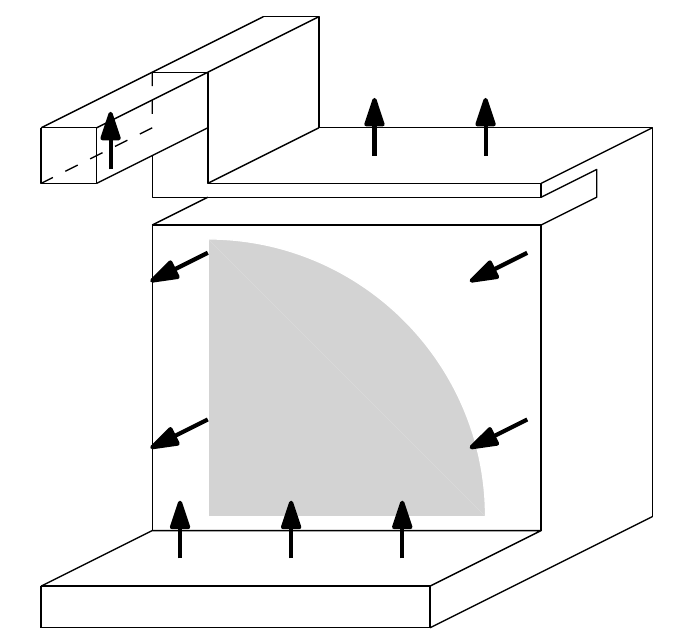}}
			\endxy
		\]
		\eqnn
			\xy
			\xyimport(8,8)(0,0){\includegraphics[width=5.5in]{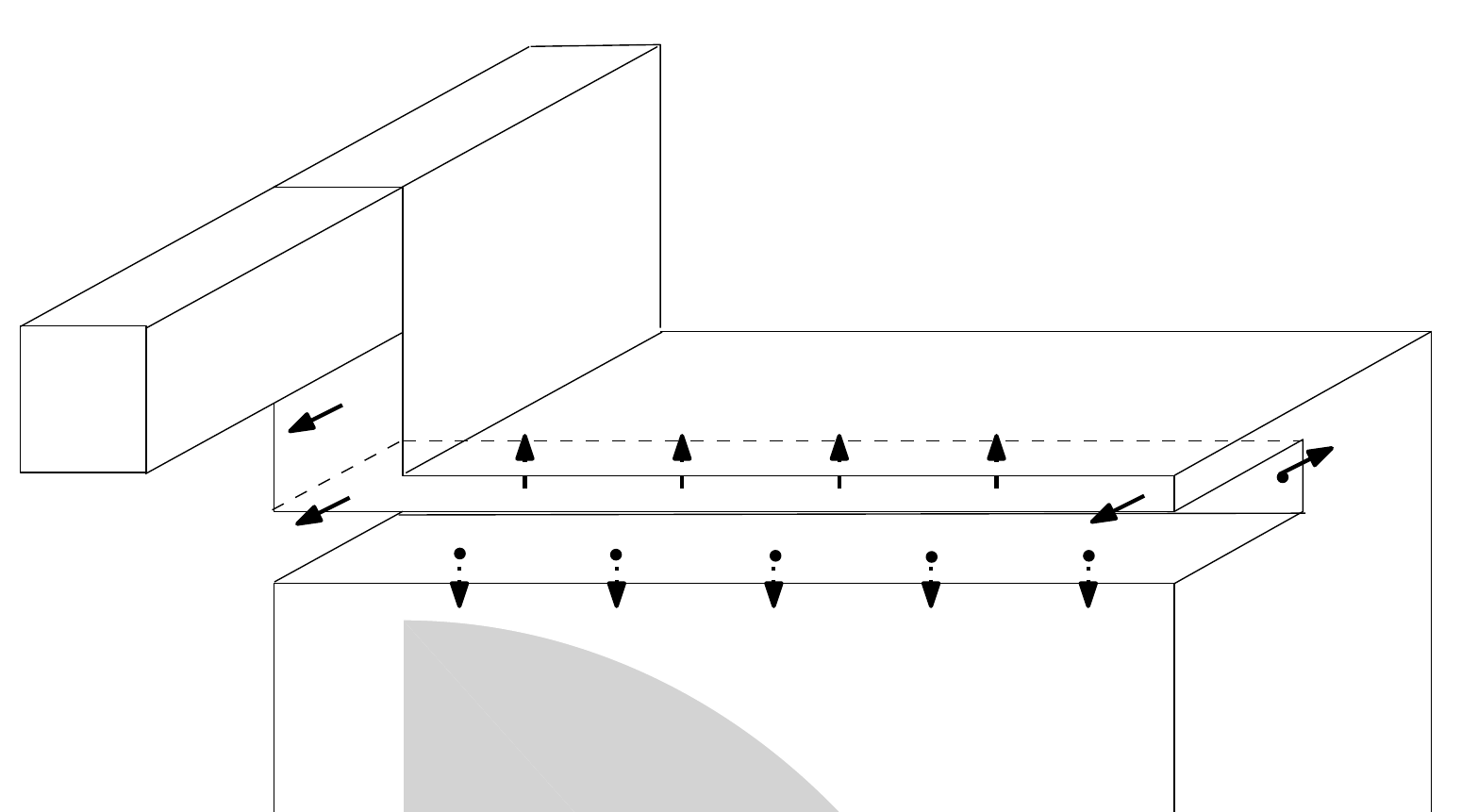}}
			\endxy
		\eqnd
\begin{image}\label{figure.T2}
On the top, an image of the cobordism defining the 2-simplex $T_2$. The obvious portion is zoomed in and displayed on the bottom. 
\end{image}
\end{figure}

\clearpage

\section{$T_1$, geometrically: The equivalence (\ref{eqn.T1-equivalence-1})}
Throughout, we fix some test object $X \in \ob \wrap$. The goal of this section is to prove

\begin{lemma}\label{lemma.T1-equivalence-1}
There is an equivalence of diagrams between $\Xi(T_1)(X)$ and the upper-right triangle in~(\ref{eqn.chain-kernel}). 
\end{lemma}

Concretely, this means one can produce a map $\Delta^2 \times \Delta^1 \to \wrap\Mod$, which we depict as follows:
	\eqn\label{eqn.T1-prism-1}
		\xymatrix{
		&&& K(X) \ar[ddlll]^{\sim} \ar[r] \ar[dr] & \Xi(L_0)(X) \ar[d]\ar[ddlll]^{\sim} \\
		&&& & \Xi(L_1)(X) \ar[ddlll]^{\sim}\\
		\ker \Xi(Y)(X) \ar[r]^p \ar[dr]^h & \Xi(L_0)(X) \ar[d]^{\Xi(Y)_X} \\
		& \Xi(L_1)(X) &&&.
		}
	\eqnd
The lower-left triangle is specified by setting $h = \Xi(Y) \circ p$ and filling the triangle with the zero homotopy. The upper-right triangle is $\Xi(T_1)$ applied to $X$.

This equivalence of diagrams is constructed by considering a Hamiltonian isotopy of $\beta$ to $\gamma$. Concretely, let $I \subset E$ be a compact interval such that $B(T_1)$ is contained in $M \times T^*I \times T^*F^2$. We let $\beta|_I = \beta \cap T^*I \subset T^*E$ and choose a compactly supported Hamiltonian on $T^*E$ such that $\beta|_I$ is taken to $\gamma|_I$.

Then the Hamiltonian suspension of this isotopy can be made into a collared Lagrangian cobordism by making the Hamiltonian time-dependent, and equal to zero outside some compact time interval. We call this cobordism $Q'$. Taking the direct product with $X \times \gamma^2$, one obtains a cobordism
	\eqnn
		Q \subset M \times T^*E \times T^*F^2 \times T^*F
	\eqnd
from $X \times \beta \times \gamma^2$ to $X \times \gamma \times \gamma^2$. Now let $\overline{\gamma} \subset T^*E$ be the reflection of $\gamma$ about the zero section. For brevity, we let
	\eqnn
		(X,T_1)_{\beta} := CF^*(X \times \beta \times \gamma^2, T_1)
		\qquad
		\text{and}
		\qquad
		(X,T_1)_{\gamma} := CF^*(X \times \gamma \times \gamma^2, T_1).
	\eqnd
Computing the Floer complex of $Q$ with $T_1 \times \overline{\gamma}$, one finds
	\eqn\label{eqn.suspension-complex}
		CF^*(Q, T_1 \times \overline{\gamma})
		\cong
		(X,T_1)_{\beta} \oplus (X,T_1)_\gamma[-1].
	\eqnd
Now let us study the differential of this complex. As in Section~\ref{section.1-simplex}, the component of the differential from the $\beta$ component to the $\gamma$ component encodes a linear map
	\eqnn
		(X,T_1)_{\beta} \to (X,T_1)_\gamma
	\eqnd
and the matrix components of this map define the simplices of~(\ref{eqn.T1-prism-1}).

\begin{figure}
		\[
			\xy
			\xyimport(8,8)(0,0){\includegraphics[height=3in]{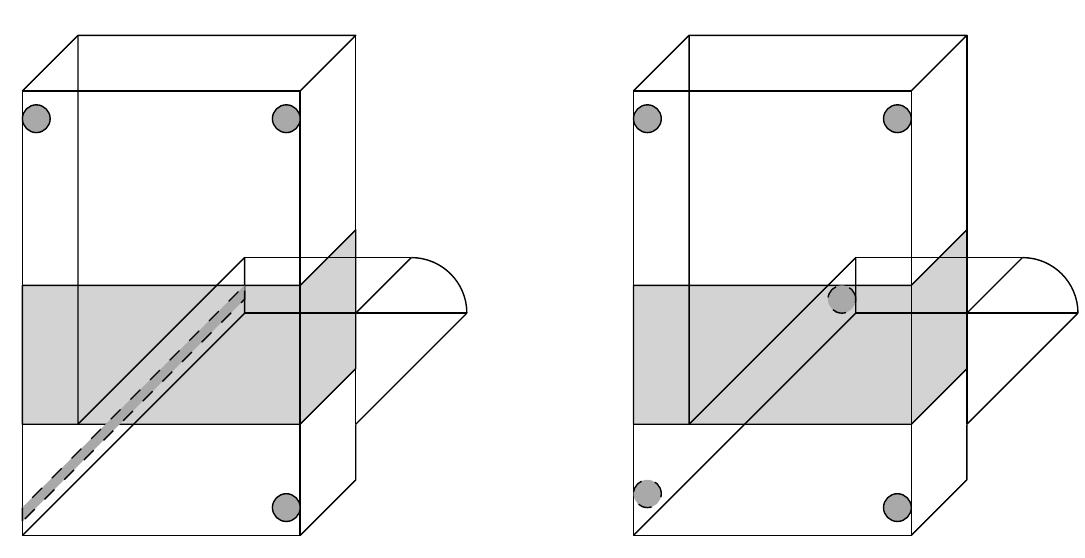}}
			,(0,0)+*{(X,\ker Y)}
			,(0,7)*+{(X,L_1[-1])}
			,(2,7)*+{(X,L_1[-2])}
			,(2,0)*+{(X,L_0[-1])}
			,(1.2, -0.5)*+{(a)}
			,(5.8, -0.5)*+{(b)}
			,(4.5,0)*+{(X,L_0)}
			,(5.7,4.2)*+{(X,L_1[-1])}
			,(4,7)*+{(X,L_1[-1])}
			,(6,7)*+{(X,L_1[-2])}
			,(6.5,0)*+{(X,L_0[-1])}
			\endxy
		\]
		\eqnn
			\xymatrix{
			\Xi(L_1)[-1] \ar[rr]^{-\id} && \Xi(L_1)[-2] \\
			\, \\
			\, \\
			\Xi \ker Y \ar[uuu]^{h'} \ar[rr]^{p'} \ar[uuurr]^{H'} && \Xi(L_0)[-1] \ar[uuu]^{f'}
			}
			\qquad
			\xymatrix{
			\Xi(L_1)[-1] \ar[rrr]  &&& \Xi(L_1)[-2] \\
			& \Xi(L_1)[-1]  \\
			\, \\
			\Xi(L_0) \ar[uuu] \ar[rrr] \ar[uur] \ar[uuurrr] &&& \Xi(L_0)[-1] \ar[uuu]
			}
		\eqnd
\begin{image}\label{figure.T1-isotopy}
In (a), a depiction of the intersection points of $X \times \beta \times \gamma^2$ with $B(T_1)$. Because we cannot give a detailed description of $X \times \beta \bigcap \ker Y$, this region is simply drawn with a dashed rectangle in (a). In (b), a depiction of the intersection points of $X \times \gamma \times \gamma^2$ with $B(T_1)$. The $\mu^1$ terms of the Floer differential induces maps between modules as depicted in the commutative diagrams. 
\end{image}
\end{figure}

First let us study the cochain complex $(X,T_1)_{\gamma}$. It is depicted in Figure~\ref{figure.T1-isotopy}(b), where its differential is also drawn. Let us determine the differential maps explicitly:
\begin{lemma}\label{lemma.T1-K-differentials}
The differential of the cochain complex $(X,T_1)_{\gamma}$ is given by the following maps:
	\eqn\label{eqn.T1-K-differentials-1}
			\xymatrix{
			\Xi(L_1)(X)[-1] \ar[rrr]^{-\id}  &&& \Xi(L_1)(X)[-2] \\
			& \Xi(L_1)(X)[-1]  \\
			\, \\
			\Xi(L_0)(X) \ar[uuu]^f \ar[rrr]^{\id} \ar[uur]^{f} \ar[uuurrr]_{0} &&& \Xi(L_0)(X)[-1] \ar[uuu]^{f}
			}
	\eqnd
Here, $f:=\Xi(Y)_X$.
\end{lemma}

\begin{proof}
By the same directionality arguments as in~\cite{tanaka-pairing}, the collaring conditions guarantee that all pseudoholomorphic strips must have non-negative partial derivatives in the $E$ and $F$ components near the generators of the Floer complex. As a result, all possible arrows must point in directions of non-decreasing $E$ and $F$ coordinates---this is why the arrows are always pointing upward or rightward (or both).

Let us first explain the diagonal map, which we claim is 0. This is because the face collaring $B(T_1)$ along $q_1<<0$ (here, $q_1 \in F_1$) is isomorphic to the Lagrangian $Y \times B(F_1)$. It was proven in~\cite{tanaka-pairing} that higher branes which are products with the zero section have no pseudoholomorphic strips that jump more than 1 degree---in the present case, this is because such a strip would be non-constant in both the $B(F_1) \subset T^*F$ component, and in the $Y$ component. Such a strip has a two-dimensional automorphism group, hence does not lie in a 0-dimensional component of the moduli space of strips. 

Now, all other maps are either $f$ or $\pm \id$, as $B(T_1)$ is collared by either $B(Y)$ or by $B(\id)$ along its edges. As for signs, they are determined the same way one concludes that $\Xi(Y)$ defines a homotopy, as in~(\ref{eqn.2-simplex-differential}).
\end{proof}

Now we simplify the diagram~(\ref{eqn.T1-K-differentials-1}). We first note that the portion of the cochain complex generated by intersection points with $q_1 <<0, q_2<<0$ (here, $q_i \in F_i$)  is precisely a copy of $\ker f$. Now let us define maps
	\eqnn
		p: \ker f \to \Xi(L_0)(X),
		\qquad
		(x_0,x_1) \mapsto x_0
	\eqnd
and
	\eqnn
		h = f \circ p: \ker f \to \Xi(L_1)(X),
		\qquad
		(x_0,x_1) \mapsto f(x_0).
	\eqnd
Then the cochain complex $(X,T_1)_\gamma$ can be re-depicted as the following:
	\eqnn
		\xymatrix{
			\Xi(L_1)(X)[-1] \ar[r]^{-\id} &\Xi(L_1)(X)[-2] \\
			\ker f \ar[u]^h \ar[r]_p & \Xi(L_0)(X)[-1] \ar[u]^f
		}
	\eqnd
which encodes the triangle
	\eqnn
		\xymatrix{
			\ker f \ar[r]^p \ar[dr]_h & \Xi(L_0)(X) \ar[d]^f \\
			& \Xi(L_1)(X).
		}
	\eqnd
Importantly, this is the lower-left triangle in~(\ref{eqn.T1-prism-1}).

Now let us analyze the chain map $(X,T_1)_{\beta} \to (X,T_1)_\gamma$. Again by directionality, this can be encoded in a lower-triangular matrix:
	\eqnn
	\Phi = (\Phi_{ij}) =
	\left(
		\begin{array}{cccc}
		\Phi_{00} & 0 & 0 & 0\\
		\Phi_{10} & \Phi_{11} & 0 & 0\\
		\Phi_{20} & 0 & \Phi_{22} & 0\\
		\Phi_{30} &  \Phi_{31} & \Phi_{32} & \Phi_{33}
		\end{array}
	\right)
	\eqnd
We emphasize that some of the entries are degree-shifting---for example, the entry $\Phi_{10}$ is a map $(X,K) \to (X,L_1)[-1]$. We don't make the domain and codomain of each entry explicit as it will be clear from context later on. At this point, the fact that $\Phi$ is a chain map is enough to deduce Lemma~\ref{lemma.T1-equivalence-1}. First, we prove

\begin{lemma}\label{lemma.phi-equivalence}
Each map $\Phi_{ii}$ is a chain equivalence.
\end{lemma}

\begin{proof}
Consider the suspension $\overline{Q}$ of the inverse Hamiltonian isotopy. By pairing $T_1 \times \gamma$ against $\overline{Q}$, we obtain a map $\Psi$. It is a 4x4 lower-triangular matrix by directionality. As we recalled in Section~\ref{section.inverses}, there is a higher cobordism exhibiting a homotopy between the composite of the two isotopies, and the identity cobordism. This results in another 4x4 lower-triangular matrix $H$ satisfying the relation
	\eqnn
	dH + Hd = \Psi \Phi - \id.
	\eqnd
The diagonal entries of this matrix show that each $\Phi_{ii}$ is a chain equivalence, as the diagonal entires $H_{ii}$ realize the relation $dH_{ii} - H_{ii} d = \Psi_{ii}\Phi_{ii} - \id$.
\end{proof}

\begin{proof}[Proof of Lemma~\ref{lemma.T1-equivalence-1}.]
We denote the differential of $CF^*(X \times \beta \times \gamma^2, T_2)$ as labeled in Figure~\ref{figure.T1-isotopy}. Then the relation $d \Phi = \Phi d$ results in the following equations:
\begin{align}
(ii)& :  d\Phi_{ii} = \Phi_{ii} d \\
(10)& :  d\Phi_{10} + \Phi_{10}d = h \Phi_{00} - \Phi_{11}h' \\
(20)& :  d\Phi_{20} + \Phi_{20}d = p\Phi_{00} - \Phi_{22}p' \\
(30)& :  d \Phi_{30} - \Phi_{30} d = \Phi_{31}h' + \Phi_{32}p' + \Phi_{33}H' + \Phi_{10} - f \Phi_{20} \\
(13)& :  d\Phi_{31} + \Phi_{31} d = -\Phi_{33} + \Phi_{11} \\ 
(23)& :  d\Phi_{32} + \Phi_{32} d = \Phi_{33}f' - f \Phi_{22}
\end{align}
Here, the $(ij)$th relation results from examining the $i$th row and $j$th column of the matrix $d\Phi - \Phi d$. Now we construct the prism. The edges are as follows:
	\eqn\label{eqn.T1-equivalence-1-prism}
		\xymatrix{
		&&& K(X) \ar[ddlll]^{\sim}_{\Phi_{00}} \ar[r]^{p'} \ar[dr]_{h'} & \Xi(L_0)(X) \ar[d]^{f'} \ar[ddlll]^{\sim}_{\Phi_{22}} \\
		&&& & \Xi(L_1)(X) \ar[ddlll]^{\sim}_{\Phi_{33}}\\
		\ker \Xi(Y)(X) \ar[r]^p \ar[dr]_h & \Xi(L_0)(X) \ar[d]^{f} \\
		& \Xi(L_1)(X) &&&.
		}
	\eqnd
From the above relations, we deduce that 
	\begin{itemize}
		\item $\Phi_{20}$ is a homotopy filling in the top rectangle, as it realizes the relation $p \Phi_{00} \sim \Phi_{22} p'$.
		\item $\Phi_{32}$ is as homotopy filling the rightmost face, as it realizes the relation $f \Phi_{22} \sim \Phi_{33} f'$.
		\item The bottom face of the prism is filled by $\Phi_{31}h' + \Phi_{10}$, which realizes the relations $h\Phi_{00} \sim \Phi_{11}h' \sim \Phi_{33}h'$.
		\item The triangular face containing $K(X)$ is filled by $H'$.
		\item Finally, the interior of the prism is filled by $\Phi_{30}$---this is the content of the $(30)$th equation.
	\end{itemize}
Lemma~\ref{lemma.phi-equivalence} shows that each $\Phi_{ii}$ is an equivalence, so we are finished.
\end{proof}

\clearpage
\section{$T_2$, geometrically: The equivalence (\ref{eqn.T2-equivalence})}\label{section.T2}
The main goal of this section is to construct the equivalence claimed in~(\ref{eqn.T2-equivalence}). After constructing a cochain complex $I_X$ (the value of a putative module $I$ on a test object $X$), we articulate this equivalence in Lemma~\ref{lemma.T2-equivalence} below. For notational simplicity, we assume that $Y$ is an unstabilized cobordism, so it is a submanifold of $M \times T^*F$. For the stabilized version, one can replace every instance of $X$ in what follows with $X \times (E^\vee)^n$. 

Consider Figures~\ref{figure.T2-continuation-1} and~\ref{figure.T2-continuation-2}. In the first of these figures is depicted the generators of the cochain complex
	\eqnn
		CF^*(X \times \beta \times \gamma^2 , B(T_2)).
	\eqnd
We draw this Floer complex as follows:
	\eqnn
		\xymatrix{
			(X,L_1)[-1] \ar[rrr]^{-\id} & &&(X,L_1)[-2] \\ 
			\, \\
			\, \\
			(X \times \beta, \ker(Y)) \ar[uuu]^{h} \ar[uuurrr]_{H}.
		}
	\eqnd
Since the total differential squares to zero, we determine that $h$ is a chain map from $(X \times \beta, \ker Y)$ to $(X,L_1)$, and that $H$ is a homotopy from $h$ to the zero map.

\begin{figure}
		\[
			\xy
			\xyimport(8,8)(0,0){\includegraphics[width=4.5in]{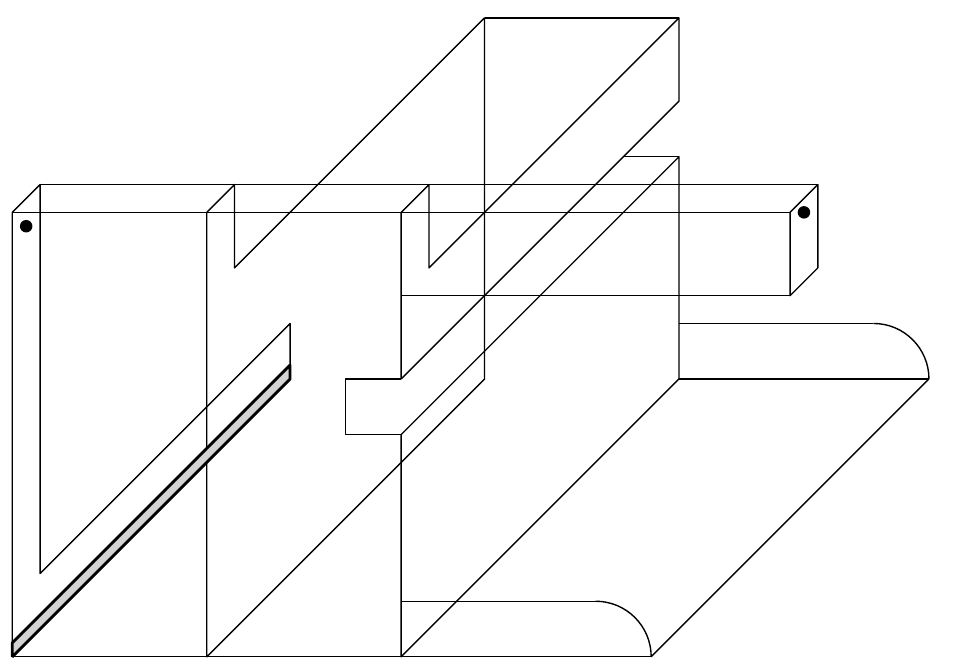}}
			,(-0.7,0.4)*+{(X,\ker Y)}
			,(-0.8,5.5)*+{(X,L_1)[-1]}
			,(7.7,5.5)*+{(X,L_1)[-2]}
			\endxy
		\]
\begin{image}\label{figure.T2-continuation-1}
An image of the intersections of $X \times \beta \times \gamma^2$ with $B(T_2)$, drawn as an image in $\RR^3 \cong E \times F^2$. Indicated in grey is the region where $X \times E$ intersects $K(Y)$. Note that this image is rotated from the one in Figure~\ref{figure.T2}.  To be explicit: Running away from the reader is $E$; the $F_1$ coordinate runs to the right, and the $F_2$ coordinate runs upward.
\end{image}
\end{figure}

\begin{figure}
		\[
			\xy
			\xyimport(8,8)(0,0){\includegraphics[width=4.5in]{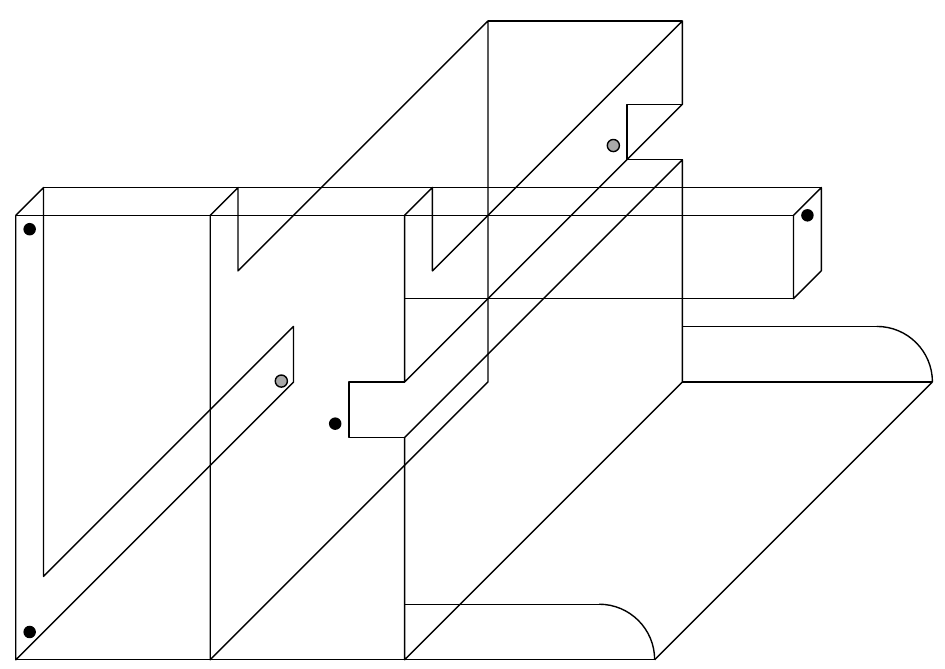}}
			,(-0.7,0.4)*+{(X,L_0)}
			,(-0.8,5.5)*+{(X,L_1)[-1]}
			,(7.7,5.5)*+{(X,L_1)[-2]}
			,(2, 3.7)*+{(X,L_1)[-1]}
			,(3.7,2.5)*+{(X,L_1)[-1]}
			,(6.5,6.5)*+{(X,L_1)[-2])}
			\endxy
		\]
\begin{image}\label{figure.T2-continuation-2}
An image of the intersections of $X \times \gamma \times \gamma^2$ with $B(T_2)$. Indicated in grey is the region where $X \times E$ intersects $K(Y)$. The intersection points live in the regions where the ``vector fields'' (i.e., the components of $(F^\vee)^3$) are arbitrary negative in all three $F^\vee$ components. The grey dots are intersection points occurring where the $E$ coordinate is large.
\end{image}
\end{figure}
	
\begin{lemma}\label{lemma.T2-gamma}
The differentials of the Floer complex 
	\eqnn
		CF^*(X \times \gamma \times \gamma^2 , B(T_2)).
	\eqnd
are as follows:
	\eqnn
		\xymatrix{
			(X,L_1)[-1] \ar[rrrrr]^{-\id} &&&&&(X,L_1)[-2] \\ 
			 & & & (X,L_1)[-2] \\
			 & (X, L_1)[-1] \ar[urr]) & (X,L_1)[-1] \ar[ur]_{\id} \ar@/_2pc/[uurrr]_{G_2}\\
			(X, L_0) \ar[uuu]^h \ar[urr] \ar[ur]^f\ar@/^2pc/[uuurrrrr]^{G_0} \ar[uurrr]
		}
	\eqnd
\end{lemma}

\begin{proof}
In Figure~\ref{figure.T2-continuation-2} is depicted the generators of the Floer complex. The content of the present Lemma is to show that differentials only propagate in a particular direction, and that we can identify the labeled arrows as $\id$, $f$, and $h$.

The directionality of the differential (for instance, there are no differentials from $(X,L_1)[-2]$ to any other summand) is the usual directionality argument---near each intersection point, we choose a direct sum almost-complex structure and a boundary-stripping argument so that the directional derivative in the $E$ or $F$ directions must be non-negative.

Now we show how to identify the differentials. We make no comment on $G_0$ and $G_2$, as these are simply arbitrary maps given by counting strips. The arrows labeled $\id$, $f$, and $h$ are identified with the labels because, in those regions, the cobordism $T_2$ is collared by the identity morphism, by $Y$, and by the map $\ker Y \to L_1$, respectively. This finishes the proof.
\end{proof}

For brevity, let us treat the bottom half of this diagram as a single cochain complex, $I_X$. Concretely, $I_X$  is the cochain complex that we draw as
	\eqn\label{eqn.I_X}
		\xymatrix{
			&&& (X,L_1)[-2] \\
			&(X,L_1)[-1] \ar[urr] & (X,L_1)[-1] \ar[ur] \\
			(X,L_0) \ar[ur] \ar[urr] \ar[rrruu] &&& .
		}
	\eqnd
Then the Floer complex 
	$
	CF^*(X \times \gamma \times \gamma^2 , B(T_2))
	$
may be summarized as
	\eqnn
		\xymatrix{
			(X,L_1)[-1] \ar[rrr]^{-\id} & &&(X,L_1)[-2] \\ 
			\, \\
			\, \\
			I_X \ar[uuu]^{g} \ar[uuurrr]_{G = G_0 \oplus G_2}
		}
	\eqnd
Since the total differential squares to zero, we determine that $g$ is a chain map from $I_X$ to $(X,L_1)$, and that $G$ is a homotopy from $g$ to the zero map.

Now, one can apply the same Hamiltonian isotopy as before to turn $\beta$ into $\gamma \subset T^*E$. As usual, pairing against the suspension of this isotopy, one finds an associated map
	\eqnn
		\phi:	CF^*(X \times \beta \times \gamma^2, B(T_2)) \to CF^*(X \times \gamma \times \gamma^2, B(T_2)).
	\eqnd
$\phi$ is represented by a lower-triangular matrix
	\eqnn
        \left(
        \begin{array}{ccc}
         \phi_{00} &  0  & 0  \\
         \phi_{10} & \phi_{11}   & 0  \\
         \phi_{20} & \phi_{21}  & \phi_{22}  
        \end{array}
        \right)
	\eqnd
where $(a_0,a_1,a_2) \in K(X) \oplus (X,L_1)[-1] \oplus (X,L_1)[-2]$ is sent to
	\eqnn
		( \sum \phi_{0i} a_i, \sum \phi_{1i} a_i, \sum \phi_{2i} a_i) \in I_X \oplus (X,L_1)[-1] \oplus (X,L_1)[-2].
	\eqnd

Exploiting the inverse Hamiltonian isotopy (as in Lemma~\ref{lemma.phi-equivalence}), one arrives at:

\begin{lemma}
Each diagonal entry $\phi_{ii}$ is a homotopy equivalence of cochain complexes.
\end{lemma}

Recall the computations that led us to the prism~(\ref{eqn.T1-equivalence-1-prism}) in the proof of Lemma~\ref{lemma.T1-equivalence-1}. Just as we did there, writing out the equation $d \phi = \phi d$ results in the following:

\begin{lemma}\label{lemma.T2-equivalence}
There is an equivalence of diagrams $\Delta^2 \times \Delta^1 \to \chain$ as follows:
	\eqn\label{eqn.T2-equivalence-prism}
		\xymatrix{
		&&& K(X) \ar[dr]^{h'} \ar[d] \ar[ddlll]_{\phi_{00}}^\sim\\
		&&&0 \ar[r] \ar[ddlll]& \Xi(L_1)(X) \ar[ddlll]_{\phi_{22}}^\sim\\
		I_X \ar[dr]^g  \ar[d]              \\
		0 \ar[r]    & \Xi(L_1)(X)  
		}
	\eqnd
The interior is filled by the element $\phi_{20} \in \hom^{-2}(K(X), \Xi(L_1)(X))$, and the five faces of the prism are given by 
	\eqnn
		0, \qquad
		0, \qquad
		G, \qquad
		H', \qquad
		\phi_{21}h' + \phi_{10}.
	\eqnd
\end{lemma}

\clearpage
\section{$T_1$, algebraically: The equivalence (\ref{eqn.T1-equivalence-2})}
The goal of this section is to exhibit the equivalence~(\ref{eqn.T1-equivalence-2}), which we encode in Lemma~\ref{lemma.T1-equivalence-2}. As before, we fix a test object $X$ throughout this section.

Before we proceed, let us label the maps involved in defining the complex $I$. We ``flatten out'' the defining arrows in~(\ref{eqn.I_X}) to make the arrows more visible:
	\eqnn
		\xymatrix{
			(X,L_1)[-1] \ar[r]^{-q} & (X,L_1)[-2] \\
			(X,L_0) \ar[r]^r \ar[u]^f \ar[ur]^{T} & (X,L_1)[-1]	\ar[u]^{\id}
		}
	\eqnd
The arrows $T, q, r$ are undetermined maps, and all we can glean from $I$ so far is that $T$ defines a homotopy from $qf$ to $r$:
	\eqn\label{eqn.dT}
		[d,T] = qf - r.
	\eqnd
For explicitness, we record that the differential of the complex $I_X$ is equal to
	\eqnn
		d(x_0, x_1, x_2, x_3)
		=
		(d x_0, fx_0 - dx_1, r x_0 - dx_2, Tx_0 - q x_1 + x_2 + d x_3).
	\eqnd

\subsection{An (eventual) universal property for $I$}
First we exhibit a property that will turn out to be the universal property of $I$ as a kernel object.

Assume that one has a diagram of cochain complexes $\Delta^1 \times \Delta^1 \to \chain$ as follows:
	\eqnn
		\xymatrix{
			A \ar[r]^k \ar[d] \ar[dr]^Q_R & (X,L_0) \ar[d]^f \\
			0 \ar[r] & (X,L_1).
		}
	\eqnd
where the triangles are labeled $Q$ and $R$. Then one has an induced map to $I_X$,
	\eqnn
		\eta: A \to I,
		\qquad
		x \mapsto (k x, Qx- Rx, -Tk x + qQx - qRx, 0).
	\eqnd

\begin{lemma}
$\eta$ is a chain map.
\end{lemma}

\begin{proof}
We compute $d \eta x$:
	\begin{align}
		d\eta x 
		&= d(k x, Qx-Rx, -Tk x + q Q x - q R x, 0) \nonumber\\
		&= (dk x, fk x - d(Qx - Rx), rk x - d(-Tk x + qQx - qRx), \nonumber\\
		 &\qquad Tk x - q(Qx-Rx) - Tk x + qQx - qRx) \nonumber \\
		&= (dk x, fk x - d(Qx - Rx), rk x - d(-Tk x + qQx - qRx), 0 ) \nonumber
	\end{align}
while
	\begin{align}
		\eta d x
		&= (k dx, Qdx - Rdx, -Tk dx + qQdx - qRxd, 0) \nonumber
	\end{align}
Analyzing $d\eta x - \eta d x$ component by component, we see
	\eqnn
		(d\eta x - \eta d x)_0 = d k x - k d x = 0
	\eqnd
since $k$ was assumed a chain map. By the definition of $Q$ and $R$, we have:
	\begin{align}
		(d\eta x - \eta d x)_1 
		&= fkx - [d,Q-R]x \nonumber \\
		&= fkx - (fk - 0) x = 0. \nonumber
	\end{align}
Using~(\ref{eqn.dT}) we find
	\begin{align}
		(d\eta x - \eta d x)_2
		&= rkx + [d,T]kx -q[d,Q-R]x\nonumber \\
		&= rkx + (qf-r)kx - q(fk-0)x \nonumber \\
		&= 0.\nonumber
	\end{align}
This completes the proof.
\end{proof}

\subsection{$\eta$ applied to $\ker f = \ker(\Xi(Y)(X))$}
Now we apply the construction of $\eta$ to $\ker f$. In this case, we know that $Q=0$, $\phi = p$, and $R(x_0,x_1) = -x_1$. (See the discussion near~(\ref{eqn.chain-kernel}).) Hence 
	\eqnn
		\eta(x_0,x_1)
		=
		(x_0, x_1, -Tx_0 +qx_1, 0).
	\eqnd

\begin{lemma}\label{lemma.eta-equivalence}
The map $\eta: \ker f \to I$ is an equivalence.
\end{lemma}

\begin{proof}
We claim we have an inverse chain map $\pi: I \to \ker f$ given by
	\eqnn
		(x_0, x_1, x_2, x_3) \mapsto (x_0, x_1).
	\eqnd
Clearly, $\pi \circ \eta = \id_{\ker f}$. As for $\eta \circ \pi$, one has a homotopy
	\eqnn
		N: I \to I,
		\qquad
		(x_0,x_1,x_2,x_3)
		\mapsto
		(0, 0, x_3, 0).
	\eqnd
For then
	\begin{align}
	dN+Nd(x_0,x_1,x_2,x_3)
	&= d(0,0,x_3,0) + (0,0,Tx_0 -qx_1 + x_2 + dx_3,0) \nonumber \\
	&= (0, 0, Tx_0 - qx_1 +x_2, x_3).\nonumber
	\end{align}
But this is obviously equal to $\id_I - \eta \circ \pi$.
\end{proof}

\begin{lemma}\label{lemma.T1-equivalence-2}
$\eta$ induces an equivalence of diagrams
	\eqn\label{eqn.T1-equivalence-2-prism}
		\xymatrix{
		&&& \ker f \ar[ddlll]^{\sim}_{\eta} \ar[r]^{p} \ar[dr]_{h} & \Xi(L_0)(X) \ar[d]^{f} \ar[ddlll]^{\sim}_{\id} \\
		&&& & \Xi(L_1)(X) \ar[ddlll]^{\sim}_{\id}\\
		I_X \ar[r]^{P} \ar[dr]_{g} & \Xi(L_0)(X) \ar[d]^{f} \\
		& \Xi(L_1)(X) &&&
		}
	\eqnd
where $P: I_X \to \Xi(L_0)(X)$ is the projection $(x_0,x_1,x_2,x_3) \mapsto x_0$.
\end{lemma}

\begin{proof}
Everything in sight commutes on the nose---all 2- and 3-simplices are specified by 0. That $g = f \circ P$ follows from examining the vertical arrow labeled ``$h$'' in Lemma~\ref{lemma.T2-gamma}.
\end{proof}

\section{Proof of Theorem~\ref{theorem}}
So far we have yielded prisms

\eqnn
	\xymatrix{
		& K(X) \ar[d] \ar[dr] \ar[dl]^{\sim}_{(\ref{eqn.T2-equivalence-prism})} \\
		I_X \ar[dr] \ar[d] & 0 \ar[r] \ar[dl] & (X,L_1) \ar[dl]^{\sim}\\
		0 \ar[r] & (X,L_1)
	}
	\text{and}
	\xymatrix{
		&&K(X) \ar[r] \ar[dr] \ar[dl]^{\sim}_{(\ref{eqn.T1-equivalence-1-prism})}
				& (X,L_0) \ar[d] \ar[dl]^{\sim}\\
		&\ker f\ar[r] \ar[dr]\ar[dl]^{\sim}_{(\ref{eqn.T1-equivalence-2-prism})}
			& (X,L_0) \ar[d] \ar[dl]^{\sim} 
			&	 (X,L_1)\ar[dl]^{\sim}\\
		I_X \ar[r] \ar[dr] & (X,L_0) \ar[d] & (X,L_1)\ar[dl]^{\sim} \\
			& (X,L_1) & &.
	}
\eqnd
If we could glue these together along a common rectangle, we would produce an equivalence~(\ref{eqn.square-equivalence}) as we desire. We show that there is such a common rectangle after a small modification, and glue the diagrams together to obtain:

\begin{lemma}\label{lemma.rectangles-glue}
There is an equivalence of diagrams as in~(\ref{eqn.square-equivalence}).
\end{lemma}

\begin{proof}
Recall that the interior of the face in (\ref{eqn.T2-equivalence-prism}) is given by the element 
	\eqnn
	\phi_{21}h' + \phi_{10} \in \hom^{-1}(K(X), (X,L_1)).
	\eqnd
Importantly, note that by definition, the component $\phi_{21}$ counts strips whose $E$ coordinate is constant and $<<0$, whose $F_2$ coordinate is constant and $>>0$, and whose $F_1$ coordinate varies. But along this region, $B(T_2)$ is identical to $B(T_1)$. Which is to say, the operation $\phi_{21}$ is {\em equal} to the operation $\Phi_{31}$---these two count continuation strips constrained to a region in which the pseudoholomorphic strip equations and their boundary conditions are identical, hence their moduli of solutions are in bijection.

By similar reasoning, the operations $\phi_{10}$ and $\Phi_{10}$ are also equal---this is because $\phi_{10}$ and $\Phi_{10}$ count precisely those strips constrained to the face $F_1 <<0$. Again, in this region, the pseudoholomorphic strip equations and their boundary conditions are equal, so their moduli of solutions are in bijection.

What we thus conclude is that the element $\phi_{21}h' + \phi_{10}$ is equal to the element 
	\eqnn
		\Phi_{21}h' + \Phi_{10}.
	\eqnd
There is one issue remaining: The composite face from $(\ref{eqn.T1-equivalence-2-prism})\circ(\ref{eqn.T1-equivalence-1-prism})$ does not have {\em edges} which are equal to the edges in~(\ref{eqn.T2-equivalence-prism}). There is exactly one problem edge: The composite morphism $\eta \circ \Phi_{00}$ is not equal to $\phi_{00}$. (The edge from $(X,L_i)$ to itself is identical for both faces, again by a collaring argument as above.)

However, note that if one composes with the projection map 
	\eqnn
	P : I_X \to \Xi(L_0)(X),
	\qquad
	(x_0,x_1,x_2,x_3) \mapsto x_0
	\eqnd
one sees that
	\eqn\label{eqn.p-equal}
		P \circ \eta \circ \Phi_{00} = P \circ \phi_{00}.
	\eqnd
To explain this, coordinatize so that $E = \{x\}$, $F_1 = \{q_1\}$, and $F_2 = \{q_2\}$. By definition, $\phi_{00}$ counts continuation strips constrained to where $q_1<<0$ is constant, and where $q_2<<0$ is also constant. ($P \circ \phi_{00}$ counts those strips additionally ending up at where $x_1<<0$.) $\Phi_{00}$ counts such strips along with other continuation strips, but $P$ picks out only those strips satisfying $q_1,q_2<<0$ (and in fact, with $x_1 <<0$ as well). The upshot is that, in this region, the defining cobordisms for $\Phi_{00}$ and $\phi_{00}$ are identical, so one once again obtains a bijection of moduli spaces.

What~(\ref{eqn.p-equal}) tells us is that the one can modify the composite prism by simply changing the edge from $K(X)$ to $I_X$ by replacing $\eta\circ\Phi_{00}$ with $\phi_{00}$:
	\eqn
		\xymatrix{
		&&& K(X) \ar[ddlll]^{\sim}_{\phi_{00}} \ar[r]^{p'} \ar[dr]_{h'} & \Xi(L_0)(X) \ar[d]^{f'} \ar[ddlll]^{\sim}_{\Phi_{22}} \\
		&&& & \Xi(L_1)(X) \ar[ddlll]^{\sim}_{\Phi_{33}}\\
		I_X \ar[r]^P \ar[dr]_g & \Xi(L_0)(X) \ar[d]^{f} \\
		& \Xi(L_1)(X) &&&.
		}
	\eqnd
This prism's bottom face is now identical to the top face of~(\ref{eqn.T2-equivalence}), so one can glue them together to finally obtain the equivalence of diagrams $(\Delta^1 \times \Delta^1) \times \Delta^1 \to \chain$.
\end{proof}

Now consider the resulting square
	\eqn\label{eqn.I-square}
		\xymatrix{
		I_X \ar[r]^P \ar[dr]^g_G \ar[d] & \Xi(L_0)(X) \ar[d]^f \\
		0 \ar[r] & \Xi(L_1)(X).
		}
	\eqnd

\begin{lemma}\label{lemma.I-pullback}
(\ref{eqn.I-square}) is a pullback diagram in $\chain$.
\end{lemma}

\begin{proof}
It suffices to prove that, if one pulls back~(\ref{eqn.I-square}) by $\eta$, one obtains the usual fiber square for $f$. So consider the two 3-simplices induced by $\eta$:
	\eqnn
		\xymatrix{
			\ker f \ar[ddr]_{\eta} \ar[ddrr]^{p} \ar[dddrr] \\
			\, \\
			& I_X \ar[dr]_g \ar[r] & \Xi(L_0)(X) \ar[d]^f \\
			&& \Xi(L_1)(X).
		}
		\qquad
		\xymatrix{
			\ker f \ar[ddr] \ar[dddr] \ar[dddrr]^h \\
			\, 			\\
			& I_X \ar[dr] \ar[d] \\
			&0 \ar[r] & \Xi(L_1)(X).
		}
	\eqnd
We call the lefthand 3-simplex $\widetilde T_1$, and the righthand 3-simplex $\widetilde T_2$. Their faces are exhibited by
\begin{itemize}
	\item $d_0 \widetilde T_1 = 0$, as $f \circ P = g$ (as we mentioned in Lemma~\ref{lemma.T1-equivalence-2}).
	\item $d_3 \widetilde T_1 = 0$, as $P \circ \eta = p$ by (the purely algebraic) definition of each map.
	\item $d_2 \widetilde T_1 = 0$, as $fp = fP\eta = g\eta$. 
	\item $d_1 \widetilde T_1 = 0$, as $fp = h$.
\end{itemize}
Of course, one can fill the 3-simplex by setting $\widetilde T_1 = 0 \in \hom^{-2}(\ker f, \Xi(L_1)(X)).$

As for $\widetilde T_2$, the faces are exhibited by
\begin{itemize}
	\item $d_0 \widetilde T_2 = G$, as $G$ exhibits the homotopy between 0 and $g$.
	\item $d_3 \widetilde T_2 = G\eta-R$. This exhibits the composite homotopy $0\eta \sim g\eta = h \sim 0$.
	\item $d_2 \widetilde T_2 = 0$, as $h=g\eta$; this also ensures that one can glue $\widetilde T_2$ to $\widetilde T_1$. 
	\item $d_1 \widetilde T_2 = R$.
\end{itemize}
Since the alternating sum of the induces faces is $G\eta - R + 0 - (G\eta - R) = 0$, this can be filled by the trivial 3-simplex.

Importantly, $d_1$ of each 3-simplex recovers the 2-simplices $0$ and $R$ defining the kernel diagram.

In short, what we have proven is the following. Consider the $\infty$-category $\sD$ of objects in $\chain$ living over the diagram 
	\eqnn
	\xymatrix{
		& \Xi(L_0)(X) \ar[d]^f \\
		0 \ar[r] & \Xi(L_1)(X).
	}
	\eqnd
We let $\underline{I_X}$ denote the object of $\sD$ exhibited by~(\ref{eqn.I-square}), and we let $\underline{\ker f}$ denote the object exhibited by~(\ref{eqn.chain-kernel}).

The proof of the present lemma shows so far that $\eta$ induces a map $\underline{\eta}: \underline{\ker f} \to \underline{I_X}$. And we know from Lemma~\ref{lemma.eta-equivalence} that $\eta$ is an equivalence; hence $\underline{\eta}$ is an equivalence in $\sD$. This completes the proof.
\end{proof}

\bibliographystyle{amsalpha}
\bibliography{biblio}

\end{document}